\documentclass[a4paper,12pt]{article}

\usepackage{xcolor}

\makeatletter
\newcommand{\globalcolor}[1]{%
  \color{#1}\global\let\default@color\current@color
}
\makeatother

\usepackage{comment}
\usepackage[utf8]{inputenc}
\usepackage[T1]{fontenc}
 \usepackage{bbold}
\usepackage{amssymb}
\usepackage{amsmath, amsthm}
\usepackage{amsfonts}
\usepackage{graphics}
\usepackage{color}
\usepackage{xcolor}
\usepackage{graphicx}
\usepackage{enumerate}
\usepackage{hyperref}
\usepackage[english]{babel}

\usepackage{caption}
\usepackage{subcaption}
\usepackage{tikz}
\usepackage{adjustbox}
\usepackage{pgfplots}
\pgfplotsset{compat=1.16}
\usetikzlibrary{shapes}
\usetikzlibrary{graphs, graphs.standard}
\usetikzlibrary{positioning}
\usetikzlibrary{scopes}
\usetikzlibrary{fit,backgrounds}
\usetikzlibrary{decorations.pathreplacing}
\usetikzlibrary{patterns,arrows.meta}

\tikzstyle{none}=[]
\tikzstyle{base}=[circle, fill=black!12, draw, inner sep=0pt, minimum width=8pt, minimum height=8pt, line width=0.5pt, draw=black]
\tikzstyle{nodeLabel}=[shape=circle, fill=white, minimum width=8pt, minimum height=8pt, inner sep=0pt]
\tikzstyle{dashEdge}=[-, dashed]
\tikzstyle{baseEdge}=[-, draw=black, line width=1pt]
\tikzstyle{wideEdge}=[-, draw=black, line width=2pt]
\tikzstyle{arrow}=[-{Latex[length=2mm]}, draw=black, line width=1pt]

\pgfdeclarelayer{edgelayer}
\pgfdeclarelayer{nodelayer}
\pgfsetlayers{edgelayer,nodelayer,main}

\definecolor{softgreen}{HTML}{2cab27}
\definecolor{sanguine}{HTML}{eb5a21}
\definecolor{softyellow}{HTML}{e6e337}
\definecolor{coolpink}{RGB}{228, 127, 226}
\definecolor{coolgreen}{RGB}{127, 211, 125}
\definecolor{coolbrown}{RGB}{204, 174, 161}
\definecolor{coolorange}{RGB}{248, 185, 126}
\definecolor{coolred}{RGB}{255, 127, 125}
\definecolor{darkcoolred}{RGB}{255, 74, 71}
\definecolor{coolblue}{HTML}{00b4d8}
\definecolor{darkcoolblue}{HTML}{0065d8}

\usepackage{cleveref}

\newtheorem{Proposition}{Proposition}

\newtheorem{thm}{Theorem}

\newtheorem{cl}{Claim}
\newtheorem{Lemma}{Lemma}

\Crefname{thm}{Theorem}{Theorems}

\title{Matroid-reachability-based  decomposition\\ into arborescences}
\author{Florian H\"orsch\\CISPA, St Ingbert, Germany,\\and\\Benjamin Peyrille, Zolt\'an Szigeti\\ Univ. Grenoble Alpes, Grenoble INP, CNRS, Laboratory  G-SCOP }
\begin{document}

\maketitle

\begin{abstract}
	The problem of matroid-reachability-based packing of  arborescences was solved by Kir\'aly. Here we solve the corresponding decomposition problem that turns out to be more complicated. The result is obtained from the solution of the more general problem of matroid-reachability-based $(\ell,\ell')$-limited packing of  arborescences where we are given a lower bound $\ell$ and an upper bound $\ell'$ on the total number of arborescences in the packing. The problem is considered  for branchings and in directed hypergraphs as well.
\end{abstract}

\section{Introduction}

Packing and Covering is an important and well-studied  subject of Combinatorial Optimization.
In graphs, packing problems consist of fitting as many non-overlapping subgraphs of a given type as possible in the input graph, while covering problems aim to cover the whole graph with such subgraphs possibly allowing overlaps. A packing which is also a covering is called a decomposition.
Some of the classic results of the area are about packing trees and packing arborescences. 
Relevant applications include  evacuation problems \cite{japan}, rigidity problems \cite{tay}, \cite{KT} and  robustness problems in networks \cite{Egy}.
While Nash-Williams \cite{NW}, and independently Tutte \cite{Tu}, characterized graphs having a packing of $k$ spanning trees, Edmonds \cite{Egy} characterized digraphs having a packing of $k$ spanning arborescences. Frank noted in \cite{FRdtda}  that the result of Nash-Williams and  Tutte  can be obtained from the result of Edmonds  via an orientation theorem. The  covering problems, covering the edge set of a graph by forests and covering the arc set of a directed graph by branchings, were solved by Nash-Williams \cite{NW2} and by Frank \cite{FRCB}. We mention that it is well-known that the previous corresponding packing and covering problems are equivalent (see Section 10 in \cite{book}). When spanning arborescences do not exist one may instead be interested in packing reachability arborescences. Kamiyama, Katoh, and Takizawa \cite{japan} gave a surprising extension of Edmonds' theorem on packing reachability arborescences. 

To solve a rigidity problem, Katoh and Tanigawa \cite{KT} introduced and solved the problem of matroid-based packing of rooted trees, in which  given a  graph and a matroid on a multiset of its vertices, we want a packing of rooted-trees  such that for every vertex $v$ of the graph, the root-set of the rooted-trees  containing $v$ forms a basis of the matroid.
The corresponding problem in directed graphs, matroid-based packing of arborescences was solved by Durand de Gevigney, Nguyen, and Szigeti \cite{DdGNSz}. We pointed out in \cite{DdGNSz} how the result of Katoh and Tanigawa \cite{KT} can be obtained from its directed counterpart given in \cite{DdGNSz} via an orientation theorem of Frank. Katoh and Tanigawa \cite{KT} also solved the  problem of matroid-based rooted tree decomposition. The  problem of matroid-based decomposition into  arborescences was not considered in \cite{DdGNSz}, we will solve it in this paper.
A common generalization of the results of Kamiyama, Katoh, and Takizawa \cite{japan} and Durand de Gevigney, Nguyen, and Szigeti \cite{DdGNSz} was given by Kir\'aly \cite{cskir}, namely a characterization of the existence of a  matroid-reachability-based packing of arborescences, where instead of the condition having the root-set of the arborescences containing any given vertex be a basis of the matroid, it must be a basis of the restriction of the matroid to the  set of vertices from which that vertex is reachable in the original directed graph. Later the result of Kir\'aly \cite{cskir} was further refined by Gao and Yang \cite{GY}. We will use this refinement to get a TDI description of the polyhedron of the subgraphs that admit a matroid-reachability-based packing of arborescences. This and the strong duality theorem allow us to solve the problem of matroid-reachability-based $(\ell,\ell')$-limited packing of  arborescences where we are given a lower bound $\ell$ and an upper bound $\ell'$ on the total number of arborescences in the packing. This in turn will easily imply the solution of the problem of matroid-reachability-based  decomposition into  arborescences. 

We mention that all these results were extended to hypergraphs, namely packing spanning hypertrees by Frank, Kir\'aly, and Kriesell \cite{fkk}, packing spanning  hyperarborescences by Frank, Kir\'aly, and Kir\'aly \cite{fkiki}, packing reachability hyperarborescences by B\'erczi and Frank \cite{BF2}, matroid-based packing of rooted hypertrees,  matroid-based packing of hyperarborescences, matroid-reachability-based packing of hyperarborescences by Fortier et al. \cite{FKLSzT}. The problems of   matroid-based    decomposition into hyperarborescences and  matroid-reachability-based decomposition into  hyperarborescences will be treated in this paper.

Along the presentation of our results, we will show how they imply the previous results of the field (see \Cref{fig:res}).
While describing  those implications, we will only show the sufficiency  as the necessity can easily be obtained directly.

\tikzstyle{theorem}=[rectangle, fill=black!10, draw, inner sep=6pt, minimum width=8pt, minimum height=8pt, line width=0.5pt, draw=black, align=center]
\begin{figure}[ht]
	\centering
	\resizebox{\textwidth}{!}{
	\begin{tikzpicture}[font=\small]
		\begin{pgfonlayer}{nodelayer}
			\node [style=theorem, line width=2.5pt] (0) at (0, 0) {\Cref{bboboiboreach}\\
	{\footnotesize $\sf M$-reachability-based $(\ell,\ell')$-limited}};

			\node [style=theorem, line width=2.5pt] (1) at (-6.5, 0) {\Cref{bboboiboreach1}\\
	{\footnotesize $\sf M$-reachability-based}\\{\footnotesize  with limited arcs}};

			\node [style=theorem, line width=2.5pt] (2) at (6.5, 0) {\Cref{bboboibo}\\
	{\footnotesize $\sf M$-based $(\ell,\ell')$ limited}};

			\node [style=theorem, line width=2.5pt] (3) at (0, 2.5) {\Cref{vouyfy2}\\
	{\footnotesize Decomposition into $\sf M$-reachability-based}};

			\node [style=theorem, line width=2.5pt] (4) at (6.5, 2.5) {\Cref{vouyfy}\\
	{\footnotesize Decomposition into $\sf M$-based}};

			\node [style=theorem] (5) at (0, -2.5) {\Cref{thmCsaba,thmGY}\\
	Kir\'aly \& Gao, Yang\\
	{\footnotesize(Complete) $\sf M$-reachability-based}};

			\node [style=theorem] (6) at (0, -5) {\Cref{reach1}\\
	Kamiyama, Katoh, Takizawa\\
	{\footnotesize Reachability with root-set}};

			\node [style=theorem] (7) at (6.5, -5) {\Cref{thmddgnsz}\\
	Durand de Gevigney, Nguyen, Szigeti\\
	{\footnotesize Complete $\sf M$-based}};

			\node [style=theorem] (8) at (13, -5) {\Cref{mbpsaori}\\
	Szigeti\\
	{\footnotesize $\sf M$-based}};

			\node [style=theorem] (9) at (6.5, -7.5) {\Cref{edmondsarborescencesmulti}\\
	Edmonds\\
	{\footnotesize Spanning with root-set}};

		\end{pgfonlayer}
		\begin{pgfonlayer}{edgelayer}
			\draw [style=arrow] (1) to (0);
			\draw [style=arrow] (0) to (3);
			\draw [style=arrow] (0) to (2);
			\draw [style=arrow] (2) to (4);
			\draw [style=arrow] (0) to (5);
			\draw [style=arrow] (5) to (6); 
			\draw [style=arrow] (2) to (7);
			\draw [style=arrow] (7) to (9);
			\draw [style=arrow, fill=none] (6) -- (0, -7.5) node {} -- (4.5, -7.5) node {};

			\draw [style=arrow, fill=none] (8) -- (13, -7.5) node {} -- (8.5, -7.5) node {};
			\draw [style=arrow, fill=none] (2) -- (13, 0) node {} -- (13, -4.125) node {};

			\draw [style=arrow, fill=none] (5) -- (4.5, -2.5) node {} -- (4.5, -4.125) node {} ;
		\end{pgfonlayer}
	\end{tikzpicture}}
	\caption{Our results (in bold) on packings of arborescences and their implications.}\label{fig:res}
\end{figure}
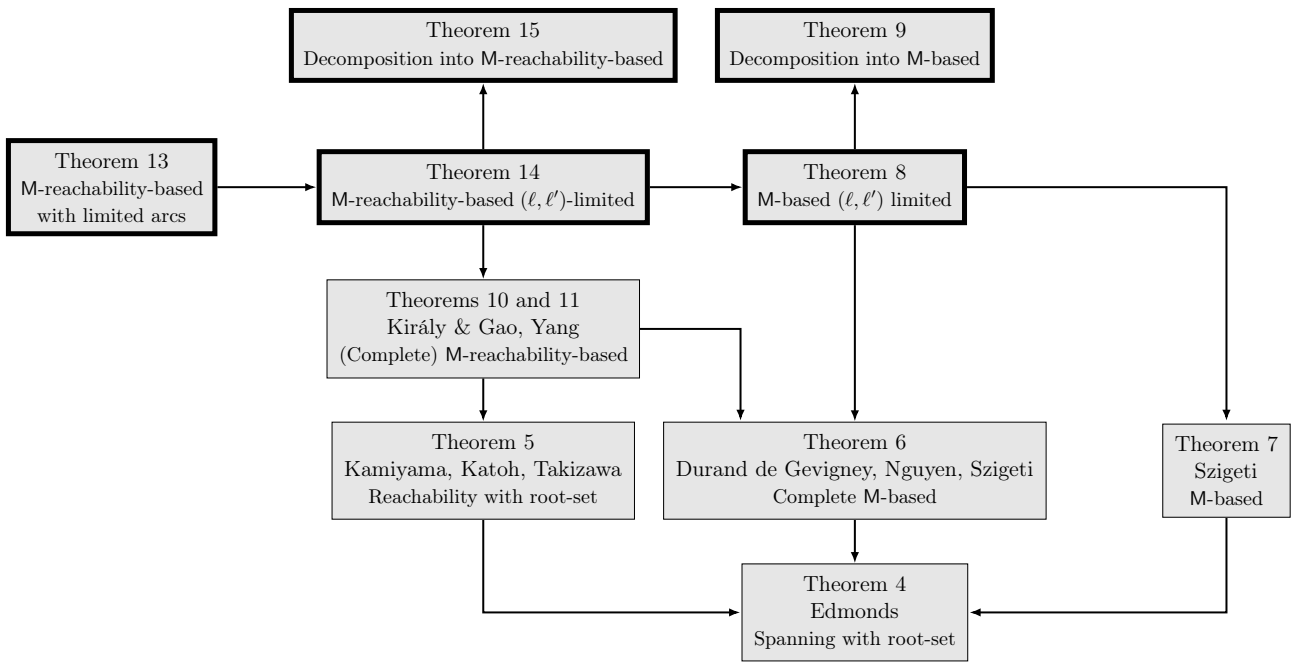












%
%
%
%
%
%

\section{Definitions}

Two subsets of  a set  $V$ are called {\it intersecting} if their intersection is non-empty. A set of mutually disjoint subsets of $V$ is called a {\it subpartition} of $V.$ For a subpartition ${\cal P}$ of $V$, {\boldmath$\cup{\cal P}$} denotes the set of elements of $V$ that belong to some member of ${\cal P}$. For a multiset $S$ of $V$ and a subset $X$ of $V$, {\boldmath$S_X$} denotes the multiset consisting of the elements of $X$ with the same multiplicities as in $S.$ For a family ${\cal S}$ of subsets of $V$ and a subset $X$ of $V$, {\boldmath${\cal S}_X$} denotes the members of ${\cal S}$ that intersect  $X.$ 
\medskip

Let $D=(V,A)$ be a directed graph, shortly {\it digraph}. A subgraph of $D$ that contains all the vertices of $D$ is called a {\it spanning subgraph} of $D.$ By a {\it packing} of subgraphs in $D$, we mean a set of subgraphs that are arc-disjoint. For a  subset $X$ of $V,$ the  {\it in-degree} of $X$, denote by {\boldmath$d^-_A(X)$}, is  the number of arcs entering $X.$ For a subpartition ${\cal P}$ of $V$, we denote by {\boldmath$e_{A}({\cal P})$} the set of arcs in $A$ that enters at least one member of ${\cal P}$. By an {\it atom} of $D$ we mean a strongly-connected component of $D,$  a {\it subatom} is a non-empty subset of an atom of $D.$
\medskip

 A digraph $(U,F)$ is called an {\it $S$-branching} if $S\subseteq U$ and there exists a unique $(S,v)$-path for every $v\in U.$ The vertex set $S$ is called the {\it root set} of the $S$-branching. If $S=\{s\},$ then the $S$-branching is an {\it $s$-arborescence} where the vertex $s$ is called the {\it root} of the $s$-arborescence. A subgraph $(U,F)$ of $D$ is called {\it reachability $s$-arborescence} if it an $s$-arborescence and $U$ is the set of vertices that can be attained from $s$ by a path in $D.$ For a subset  $X$ of $V,$ we denote by {\boldmath$P_X$} or {\boldmath$P^D_X$}  the set of vertices from which there exists a path to  at least one vertex of $X$ in $D.$ For $\ell, \ell'\in\mathbb{Z}_+$, a packing of branchings is {\it $(\ell,\ell')$-limited}  if the total number of the arborescences  in the packing is at least $\ell$ and at most $\ell'.$
\medskip

A set function $r$ on  a set  $V$ is called {\it monotone} if for all $X\subseteq Y\subseteq V,$ we have $r(X)\le r(Y).$
We say that $r$ is {\it subcardinal} if $r(X)\le |X|$ for every $X\subseteq V.$
Set functions $b$ and $p$ on $V$ are called {\it submodular} and {\it supermodular} if for all $X,Y\subseteq V,$   \eqref{submod} and \eqref{supermod} hold, respectively. We say that $b$ and $p$ are  {\it intersecting submodular} and {\it intersecting supermodular} if for all intersecting subsets $X$ and $Y$ of $V,$  \eqref{submod} and \eqref{supermod} hold, respectively.
\begin{eqnarray}
b(X)+b(Y) 		& 	\geq  	& 	b(X\cap Y)+b(X\cup Y),	\label{submod}\\
p(X)+p(Y) 		& 	\leq  		& 	p(X\cap Y)+p(X\cup Y).	\label{supermod}
\end{eqnarray}

Let $S$ be a finite ground set and $r:S\rightarrow \mathbb Z_+$ a non-negative integer valued function on $S$ such that $r(\emptyset)=0,$ $r$ is subcardinal, monotone and submodular. Then {\boldmath${\sf M}$} $=(S,r)$ is called a {\it matroid}. The function $r$ is  the {\it rank function} of the matroid ${\sf M}.$ For a matroid ${\sf M},$ its rank function will be denoted by {\boldmath$r_{\sf M}$}. An {\it independent set} of ${\sf M}$ is a subset $X$ of $S$ such that $r_{\sf M}(X)=|X|.$ The set of independent sets of ${\sf M}$ is denoted by {\boldmath${\cal I}_{\sf M}$}. A maximal independent set of ${\sf M}$ is called a {\it basis} of ${\sf M}$. Every basis of ${\sf M}$ has size $r_{\sf M}(S)$. For a subset $S'$ of $S$, a maximal independent set in $S'$ is called a {\it basis} of $S'.$ We say that two elements $s$ and $s'$ of $S$ are {\it parallel} if $r_{\sf M}(s)=r_{\sf M}(s')=r_{\sf M}(\{s,s'\}).$
The {\it free matroid} on $S$ is the matroid where the only basis is the  ground set $S.$ For a given partition $\mathcal{P}$ of $S$ and a positive integer $a_i$ for every member $X_i$ of $\mathcal{P},$ the {\it partition matroid} ${\sf M}_\mathcal{P}^a$ is the matroid whose rank function is $r_{{\sf M}_\mathcal{P}^a}(S')=\sum_{X_i\in\mathcal{P}}\min\{|S'\cap X_i|, a_i\}$ for every $S'\subseteq S.$
\medskip

In  a directed graph $D=(V,A)$, let $S$ be a multiset of $V$ and ${\sf M}$ a matroid on $S.$ A packing $\mathcal{B}$ of arborescences in $D$  is called {\sf M}-{\it based} or {\it  matroid-based} if every $s\in S$ is the root of at most one arborescence in the packing and for every vertex $v\in V$, the multiset {\boldmath$R^{\mathcal{B}}_v$} of the roots of the arborescences containing $v$ in the packing forms a basis of ${\sf M}.$ A packing $\mathcal{B}$ of arborescences in $D$  is called {\sf M}-{\it reachability-based} or {\it matroid-reachability-based}  if every $s\in S$ is the root of at most one arborescence in the packing and for every vertex $v\in V$, the multiset $R^{\mathcal{B}}_v$ of the roots  of the arborescences containing $v$ in the packing  forms a basis of $S_{P^D_v}$ in ${\sf M}.$ A packing of arborescences is {\it complete} if every $s\in S$ is the root of exactly one arborescence in the packing.
\medskip

A biset {\boldmath${\sf X}$} on a set $V$ is an ordered pair $(X_O,X_I)$ of subsets of $V$ such that $X_I\subseteq X_O.$ 
We call $X_O$ and $X_I$ the {\it outer set} and the {\it inner set} of ${\sf X}$, while {\boldmath$X_W$} $=X_O-X_I$ is called the {\it wall} of ${\sf X}$. We say that an arc $uv$  {\it enters} a biset {\sf X} if $u\in V-X_O$ and $v\in X_I$. The set of arcs in $A$ entering a biset {\sf X} is denoted by {\boldmath$\delta^-_A({\sf X})$} and {\boldmath${\sf d}^-_A({\sf X})$}$=|\delta^-_A({\sf X})|$.
For two bisets ${\sf X}=(X_O,X_I)$ and ${\sf Y}=(Y_O,Y_I),$ the {\it intersection} {\boldmath${\sf X}\cap {\sf Y}$} of ${\sf X}$ and ${\sf Y}$ is the biset $(X_O\cap Y_O,X_I\cap Y_I)$ and the {\it union} {\boldmath${\sf X}\cup {\sf Y}$} of ${\sf X}$ and ${\sf Y}$ is the biset $(X_O\cup Y_O,X_I\cup Y_I).$ 
A function on bisets is called a {\it biset function}. Biset function ${\sf p}$ is called {\it positively intersecting supermodular} if  \eqref{supermodbiset} holds for all  bisets {\sf X} and {\sf Y} with   ${\sf p}({\sf X}), {\sf p}({\sf Y})>0$ and $X_I\cap Y_I\neq\emptyset$.
\begin{eqnarray}
{\sf p}({\sf X})+{\sf p}({\sf Y}) 	& 	\leq  	& 	{\sf p}({\sf X}\cap {\sf Y})+{\sf p}({\sf X}\cup {\sf Y}).	\label{supermodbiset}
\end{eqnarray}

In  a directed graph $D=(V,A)$, we call a vertex set $Z$ a {\it petal} if there exists an atom $C$ of $D$ such that $Z\cap C\neq\emptyset, Z\subseteq P_C$ and $d_A^-(Z-C)=0.$ Note that if $Z$ is a petal, then the atom $C_Z$ is uniquely defined and $P_{C_Z}=P_Z.$  The {\it core} of a petal $Z$ is $C_Z\cap Z$. Let {\boldmath$\hat{\mathcal{Z}}_D$}, shortly {\boldmath$\hat{\mathcal{Z}}$}, be the set of petals in $D.$ We define the biset {\boldmath${\sf X}_Z$}  $=(Z,C_Z\cap Z)$ for every petal $Z$ and we call it a {\it petal biset}. Let {\boldmath$\hat{\mathcal{Z}}_{\sf b}$} be the set of petal bisets, that is $\hat{\mathcal{Z}}_{\sf b}=\{{\sf X}_Z: Z\in\hat{\mathcal{Z}}\}$. Note that for every ${\sf X}\in\hat{\mathcal{Z}}_{\sf b}$, $X_I$ is the core of the petal $X_O.$ More generally, let {\boldmath${\mathcal{X}}$} be the set of bisets {\sf X} on $V$, called {\it generalized petal bisets}, such that $X_O$ is a petal and $X_I$ is a non-empty {\it subset} of the core of the petal $X_O.$ Two petals $Z$ and $Z'$ are called {\it core-intersecting} if their cores intersect  (see \Cref{fig:petal-core-int}). Note that two petals may intersect without being core-intersecting (see \Cref{fig:petal-int-non-core-int}).
We say that a set  $\mathcal{Z}$ of petals is {\it core-laminar} if for all core-intersecting  $Z,Z'\in\mathcal{Z}$, we have $Z\subseteq Z'$ or $Z'\subseteq Z.$ More generally, bisets ${\sf X}^1, {\sf X}^2\in{\mathcal{X}}$ are called {\it core-intersecting} if their petals $X^1_O$ and $X^2_O$ are core-intersecting. 
We say that  $\mathcal{P}\subseteq{\mathcal{X}}$  is {\it OW-laminar} if for all core-intersecting ${\sf X}^1,{\sf X}^2\in\mathcal{P}$, we have $X^1_O\subseteq X^2_W$ or $X^2_O\subseteq X^1_W.$ Note that a biset on an atom is an element {\sf X} of $\mathcal{X}$ for which the petal $X_O$ of {\sf X} coincide with the core of $X_O.$

\medskip

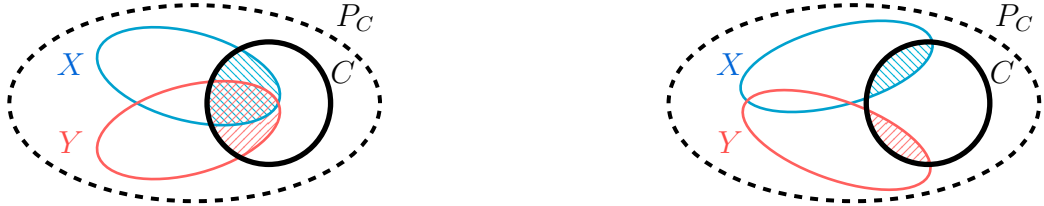
\begin{figure}[h]
	\centering
	\begin{subfigure}[b]{0.38\textwidth}
		\centering
		\begin{tikzpicture}[scale=0.65]
				\begin{scope}
					\clip (1.5,0) circle [radius=1.25];
					\pattern [pattern={north west lines},pattern color=coolblue!100,rotate around={-15:(-0.125,0.55)}] (-0.125,0.55) ellipse (1.9cm and 0.9cm);
				\end{scope}
				\begin{scope}
					\clip (1.5,0) circle [radius=1.25];
					\pattern [pattern={north east lines},pattern color=coolred!100,rotate around={15:(-0.125,-0.55)}] (-0.125,-0.55) ellipse (1.9cm and 0.9cm);
				\end{scope}

				\draw [draw=coolblue!125, line width=1pt, rotate around={-15:(-0.125,0.55)}] (-0.125,0.55) ellipse (1.9cm and 0.9cm);
				\draw [draw=coolred!125, line width=1pt, rotate around={15:(-0.125,-0.55)}] (-0.125,-0.55) ellipse (1.9cm and 0.9cm);
				\draw [draw=black, line width=2pt] (1.5,0) circle [radius=1.25];

				\draw [line width=1.5pt, dashed] (0,0) ellipse (3.75cm and 2cm);
				\draw [style=none] (3.25,1.75) node {$P_C$};
				\draw [style=none] (3.0,0.625) node {$C$};
				\draw [style=none] (-2.5,0.8) node {\color{darkcoolblue}$X$};
				\draw [style=none] (-2.5,-0.8) node {\color{darkcoolred}$Y$};
		\end{tikzpicture}
		\caption{Two core-intersecting petals}
		\label{fig:petal-core-int}
	\end{subfigure}
	\hfill
	\begin{subfigure}[b]{0.60\textwidth}
		\centering
		\begin{tikzpicture}[scale=0.65]
			\begin{scope}
			\end{scope}
				
				\begin{scope}
					\clip (1.5,0) circle [radius=1.25];
					\pattern [pattern={north west lines},pattern color=coolblue!100,rotate around={15:(-0.35,0.75)}] (-0.35,0.75) ellipse (2cm and 0.8cm);
				\end{scope}
				\begin{scope}
					\clip (1.5,0) circle [radius=1.25];
					\pattern [pattern={north east lines},pattern color=coolred!100,rotate around={-20:(-0.35,-0.75)}] (-0.35,-0.75) ellipse (2cm and 0.8cm);
				\end{scope}

				\draw [draw=coolblue!125, line width=1pt, rotate around={15:(-0.35,0.75)}] (-0.35,0.75) ellipse (2cm and 0.8cm);
				\draw [draw=coolred!125, line width=1pt, rotate around={-20:(-0.35,-0.75)}] (-0.35,-0.75) ellipse (2cm and 0.8cm);
				\draw [draw=black, line width=2pt] (1.5,0) circle [radius=1.25];

				\draw [line width=1.5pt, dashed] (0,0) ellipse (3.75cm and 2cm);
				\draw [style=none] (3.25,1.75) node {$P_C$};
				\draw [style=none] (3.0,0.625) node {$C$};
				\draw [style=none] (-2.5,0.8) node {\color{darkcoolblue}$X$};
				\draw [style=none] (-2.5,-0.8) node {\color{darkcoolred}$Y$};
		\end{tikzpicture}
		\caption{Two intersecting petals that are not core-intersecting}
		\label{fig:petal-int-non-core-int}
	\end{subfigure}
	\caption{Examples of petals. The dashed areas correspond to the cores of the petals.}
\end{figure}




\section{Total dual integrality}

The solution of our problems will rely on the polyhedral description of the subgraphs of a given digraph $D,$ that admit an ${\sf M}$-reachability-based packing of arborescences. We will use a TDI description of the polyhedron in question. To be able to do that we need some properties of TDI systems.
\medskip

A linear system $Ax\le b$ where $A$ is a rational matrix and $b$ a rational vector is called {\it Totally Dual Integral} ({\bf TDI}) if the dual linear program of $\max\{c^Tx: Ax\le b\}$ has an integer-valued optimal solution for every integral vector $c$ for which the dual has a feasible solution. 
\medskip

The seminal result of Edmonds, Giles \cite{TDI} on TDIness is the following.

\begin{thm} [Edmonds, Giles \cite{TDI}, Corollary 22.1b in \cite{lexbook}]\label{EG}
Let $Ax\le b$ be a TDI-system where $A$ is an integral matrix and $b$ is an integral vector. If $\max\{c^Tx: Ax\le b\}$ is finite, then it has an integral optimal solution.
\end{thm}

The TDI description of the polyhedron we are interested in uses bisets. The following result result of Frank \cite{frtdithm} will hence play an important role.

\begin{thm}[Frank, Theorem 5.3 in \cite{frtdithm}]\label{FJTDI}  Let $D = (V, A)$ be a digraph and ${\sf p}$ a positively intersecting supermodular biset function on $V$ such that ${\sf d}_A^-({\sf X}) \ge {\sf p}({\sf X})$  for every biset {\sf X} on $V.$
 Then the following linear system is TDI:
\begin{eqnarray*}
	x(\delta_A^-({\sf X}))&\ge& {\sf p}({\sf X}) \hskip 1truecm \text{ for every biset  {\sf X} on $V$},\\
	\mathbb{1}	\hskip .25truecm 	\ge 	\hskip .25truecm	x &\ge & \mathbb{0}.
\end{eqnarray*}
\end{thm}

We also need the following simple observation on TDI systems given by Schrijver in \cite{lexbook}.

\begin{thm}[Schrijver, (41) in \cite{lexbook}]\label{tdiclaim}
If $A_1x\le b_1$ and  $A_2x\le b_2$ define the same polyhedron, and each inequality of  $A_1x\le b_1$ is a non-negative integral combination of inequalities in  $A_2x\le b_2$, then  $A_1x\le b_1$ is TDI implies that  $A_2x\le b_2$ is TDI.
\end{thm}

\section{Packings in directed graphs}

In this section we consider packings of arborescences in digraphs. We present known and new results on packing spanning arborescences, packing reachability arborescences, matroid-based packing of arborescences, and matroid-reachability-based packing of arborescences in different subsections.

\subsection{Packing  arborescences}

Let us start with the seminal result of Edmonds on packing spanning arborescences.

\begin{thm}[Edmonds \cite{Egy}]\label{edmondsarborescencesmulti} 
Let $D=(V,A)$ be a digraph and $S$ a multiset of vertices in $V.$
There exists a packing of spanning $s$-arborescences $(s\in S)$ in $D$ if and only if 
	\begin{eqnarray} \label{edmondscondmulti} 
		|S_X|+ d^-_A(X) &\geq &|S| \hskip .5truecm \text{ for every non-empty $X\subseteq V.$}
	\end{eqnarray}
\end{thm}

It is worth mentioning that to have  a packing of spanning $s$-arborescences $(s\in S)$ in $D$ it is sufficient that \eqref{edmondscondmulti} holds for every subatom.
\medskip

The following  nice extension of Theorem \ref{edmondsarborescencesmulti} about packing of reachability arborescences was given in \cite{japan}.

\begin{thm}[Kamiyama, Katoh, Takizawa \cite{japan}]\label{reach1} 
Let $D=(V,A)$ be a digraph and $S$ a multiset of vertices in $V.$
There exists a packing of reachability $s$-arborescences $(s\in S)$ in $D$ if and only if 
	\begin{eqnarray} \label{reach1cond} 
		 |S_{X}|+d^-_A(X) &\geq& |S_{P^D_X}| \hskip .5truecm \text{ for every  $X\subseteq V.$}
	\end{eqnarray}
\end{thm}

Theorem \ref{reach1} implies Theorem \ref{edmondsarborescencesmulti} because \eqref{edmondscondmulti} implies that  each reachability $s$-arborescence is spanning and it also implies that \eqref{reach1cond} holds. 
We mention that H\"orsch and Szigeti \cite{HSz4} pointed out that Theorem \ref{reach1} can be obtained from Edmonds' result on packing spanning branchings (see Theorem \ref{edmondsbranchings}) by an easy induction.

\subsection{Matroid-based packing of arborescences}\label{mbpoa}

 The directed counterpart of the problem of matroid-based packing of rooted trees of Katoh and Tanigawa \cite{KT}, the problem of  matroid-based packing of arborescences, was solved in \cite{DdGNSz}.

\begin{thm}[Durand de Gevigney, Nguyen, Szigeti \cite{DdGNSz}] \label{thmddgnsz}
Let $D=(V,A)$ be a  digraph, $S$ a multiset of vertices in $V$, and ${\sf M}=(S,\mathcal{I}_{\sf M})$ a matroid with rank function $r_{\sf M}$. 
There exists a complete {\sf M}-based packing of arborescences in $D$  if and only if   
	\begin{eqnarray}
		S_v		&\in&	 \mathcal{I}_{\sf M} \hskip 1truecm \text{ for every } v\in V, \label{matcondori1}\\
		r_{{\sf M}}(S_Z)+d^-_{A}(Z)&\geq &r_{{\sf M}}(S) \hskip .4truecm \text{ for every non-empty } Z\subseteq V.\label{ddgnszcond}
	\end{eqnarray}
\end{thm}
\noindent For the free matroid, Theorem \ref{thmddgnsz} reduces to Theorem \ref{edmondsarborescencesmulti}.
We mention that 
\eqref{ddgnszcond} is equivalent to 
	\begin{eqnarray}
	r_{\sf M}(S_{X})+d_A^-(X)	&	\ge	&	r_{\sf M}(S)  \hskip .2truecm \text{ for every subatom } X \text{ of } D.\label{kjbvgc}
	\end{eqnarray}
Indeed,  \eqref{ddgnszcond} trivially implies \eqref{kjbvgc}. To see that \eqref{kjbvgc} implies \eqref{ddgnszcond}, let $Z$ be a non-empty subset of $V.$ Let $V_1,\dots, V_t$ be a topological ordering of the atoms of $D$  that is if   an arc exists from $V_i$ to $V_j,$ then $i<j.$ Let $i$ be the smallest index such that $V_i\cap Z\neq\emptyset.$ Since $Z\neq\emptyset,$ $i$ exists. Then $X=V_i\cap Z$ is a subset of the atom $V_i$, so $X$ is a subatom. Since we have a topological ordering, every arc entering $X$ enters $Z$, so we have  $d^-(X)\le d^-(Z).$ Further, by $X\subseteq Z$ and the monotonicity of $r_{\sf M},$ we have $r_{\sf M}(S_{X})\le r_{\sf M}(S_{Z}).$ Then, by \eqref{kjbvgc}, we get $r_{\sf M}(S) \le r_{\sf M}(S_{X})+d_A^-(X)\le  r_{\sf M}(S_{Z})+d_A^-(Z),$ so \eqref{ddgnszcond} holds.
\medskip

In the previous theorem we were interested in a matroid-based packing of arborescences containing the largest possible number of arborescences, namely $|S|$. In the following result we consider a matroid-based packing of arborescences containing the smallest possible number of arborescences, namely $r_{\sf M}(S)$.

\begin{thm}[Szigeti \cite{szigrooted}] \label{mbpsaori}
Let $D=(V,A)$ be a digraph, $S$ a multiset of vertices in $V$ and ${\sf M}=(S,r_{\sf M})$ a matroid. There exists an ${\sf M}$-based packing of  spanning arborescences in $D$  if and only if 
\begin{eqnarray}
r_{{\sf M}}(S_{\cup{\cal P}})+e_A({\cal P})&\geq &r_{{\sf M}}(S)|\mathcal{P}| \hskip .5truecm \text{ for every subpartition $\mathcal{P}$ of $V.$} \label{mbpsaoricond}
\end{eqnarray}
\end{thm}

For the free matroid, Theorem \ref{mbpsaori} reduces to Theorem \ref{edmondsarborescencesmulti}.

\subsubsection{New results on matroid-based packing of arborescences}

Here we propose to study a  problem that is more general than the above two problems. We give the following common generalization of Theorems \ref{thmddgnsz} and \ref{mbpsaori} where we have a lower bound and an upper bound on the number of arborescences in the packing.

\begin{thm} \label{bboboibo}
Let $D=(V,A)$ be a  digraph, $S$ a multiset of vertices in $V$, $\ell, \ell'\in\mathbb{Z}_+$, and ${\sf M}=(S,r_{\sf M})$ a matroid. There exists an ${\sf M}$-based $(\ell,\ell')$-limited packing of arborescences  in $D$ if and only if \eqref{kjbvgc} holds and
\begin{eqnarray} 
	\ell	&	\le	&	\ell',	\label{ellell}		\\
		\sum_{v\in V}r_{\sf M}(S_v)	&	\ge	&	\ell,	\label{jbuoouou}	\\
	\sum_{{\sf X}\in\mathcal{P}}(r_{\sf M}(S)-r_{\sf M}(S_{X_W})-d_A^-(X_O)) 	&	\le	&	\ell'	 
	\hskip .2truecm  \text{ for every OW laminar biset family } {\cal P} \text{ of subatoms}.\hskip .8truecm 	\label{dknjkbduzv}
\end{eqnarray}
\end{thm}

\begin{figure}[ht]
	\centering
	\begin{subfigure}[b]{0.5125\textwidth}
	\centering
	\begin{tikzpicture}[scale=0.625]
		\begin{pgfonlayer}{nodelayer}
			\node [style=base] (0) at (0, 0) {};
			\node [style=none] (1) at (0, 0.5) {$s_2$};
			\node [style=base] (2) at (-1.5, -1.5) {};
			\node [style=base] (3) at (1.5, -1.5) {};
			\node [style=base] (4) at (3, -3) {};
			\node [style=base] (5) at (-3, -3) {};
			\node [style=base] (6) at (-3, 3) {};
			\node [style=base] (7) at (3, 3) {};
			\node [style=none] (8) at (-3.5, 3.25) {$s_4$};
			\node [style=none] (9) at (3.5, 3.25) {$s_3$};
			\node [style=none] (10) at (-1, -1.75) {$s_1$};
			\node [style=none] (11) at (1, -1.75) {$s_1'$};
			\node [style=none] (12) at (-3.75, -3.75) {};
			\node [style=none] (13) at (-0.75, -0.75) {};
			\node [style=none] (14) at (3.75, -3.75) {};
			\node [style=none] (15) at (0.75, -0.75) {};
			\node [style=none] (16) at (0, 1) {};
			\node [style=none] (17) at (-4, -4) {};
			\node [style=none] (18) at (4, -4) {};
			\node [style=none] (19) at (-4.75, 3.25) {};
			\node [style=none] (20) at (-4.75, -4.5) {};
			\node [style=none] (21) at (4.75, -4.5) {};
			\node [style=none] (22) at (4.75, 3.25) {};
			\node [style=none] (23) at (4, 3) {};
			\node [style=none] (24) at (-4, 3) {};
		\end{pgfonlayer}
		\begin{pgfonlayer}{edgelayer}
			\draw [line width=2pt, draw=coolorange!120] (13.center)
				 to [in=-45, out=-45] (12.center)
				 to [in=135, out=135] cycle;
			\draw [line width=1.5pt, draw=coolorange!150, fill=coolorange!10] (-1.5, -1.5) circle (0.80);

			\draw [line width=2pt, draw=coolblue!120] (15.center)
				 to [in=45, out=45] (14.center)
				 to [in=-135, out=-135] cycle;
			\draw [line width=1.5pt, draw=coolblue!150, fill=coolblue!10] (1.5, -1.5) circle (0.80);

			\draw [line width=2pt, draw=coolpink!140] (17.center)
				 to [in=-135, out=-45, looseness=0.50] (18.center)
				 to [in=0, out=45] (16.center)
				 to [in=135, out=180] cycle;
			\draw [line width=1.5pt, draw=coolpink!150, fill=coolpink!10] (0,0) circle (0.80);

			\draw [line width=2pt, draw=coolred!120] (19.center)
				 to [in=120, out=-120, looseness=0.50] (20.center)
				 to [in=-120, out=-60, looseness=0.50] (21.center)
				 to [in=-60, out=60, looseness=0.50] (22.center)
				 to [in=60, out=120, looseness=0.50] cycle;
			\draw [line width=1.5pt, draw=coolred!130, fill=coolred!10] (23.center)
				 to [in=-90, out=-90, looseness=0.50] (24.center)
				 to [in=90, out=90, looseness=0.50] cycle;

			\draw [style=arrow] (0) to (6);
			\draw [style=arrow] (0) to (7);
			\draw [style=arrow] (0) to (2);
			\draw [style=arrow] (0) to (3);
			\draw [style=arrow, bend left] (2) to (0);
			\draw [style=arrow, bend right] (2) to (0);
			\draw [style=arrow, bend right] (3) to (0);
			\draw [style=arrow, bend left] (3) to (0);
			\draw [style=arrow, bend left, looseness=0.75] (6) to (7);
			\draw [style=arrow] (7) to (6);
			\draw [style=arrow] (4) to (3);
			\draw [style=arrow, bend right=60] (4) to (3);
			\draw [style=arrow, bend left=60] (4) to (3);
			\draw [style=arrow, bend right] (3) to (4);
			\draw [style=arrow, bend left] (3) to (4);
			\draw [style=arrow] (5) to (2);
			\draw [style=arrow, in=285, out=-15] (5) to (2);
			\draw [style=arrow, bend left=60] (5) to (2);
			\draw [style=arrow, bend right] (2) to (5);
			\draw [style=arrow, bend left] (2) to (5);
			\draw [style=arrow, bend right=45, looseness=0.75] (5) to (4);
			\draw [style=arrow, bend left=75, looseness=0.75] (4) to (5);
			\draw [style=arrow, bend right] (6) to (5);
			\draw [style=arrow, bend left] (7) to (4);
			\draw [style=arrow, bend right=45] (4) to (7);
			\draw [style=arrow, bend right=15] (4) to (7);
			\draw [style=arrow, bend left=15] (5) to (6);
			\draw [style=arrow, bend left=45] (5) to (6);
		\end{pgfonlayer}
	\end{tikzpicture}
	
	\caption{Family giving a lower bound of $5$ arborescences.}
	\label{fig:mbased-packing-example-bisets}
	
	\end{subfigure}
	\begin{subfigure}[b]{0.475\textwidth}
	\centering
	\begin{tikzpicture}[scale=0.75]
		\begin{pgfonlayer}{nodelayer}
			\node [style=base, fill=coolpink!100] (0) at (0, 0) {};
			\node [style=none] (1) at (0, 0.5) {$s_2$};
			\node [style=base, fill=coolgreen!70] (2) at (-1.5, -1.5) {};
			\node [style=base, fill=coolorange!70] (3) at (1.5, -1.5) {};
			\node [style=base] (4) at (3, -3) {};
			\node [style=base] (5) at (-3, -3) {};
			\node [style=base, fill=coolred!100] (6) at (-3, 3) {};
			\node [style=base, fill=coolblue!70] (7) at (3, 3) {};
			\node [style=none] (8) at (-3.5, 3.25) {$s_4$};
			\node [style=none] (9) at (3.5, 3.25) {$s_3$};
			\node [style=none] (10) at (-1, -1.75) {$s_1$};
			\node [style=none] (11) at (1, -1.75) {$s_1'$};
		\end{pgfonlayer}
		\begin{pgfonlayer}{edgelayer}
			\draw [style=arrow, draw=coolpink!140] (0) to (6);
			\draw [style=arrow, draw=coolpink!140] (0) to (7);
			\draw [style=arrow, draw=coolpink!140] (0) to (2);
			\draw [style=arrow, draw=coolpink!140] (0) to (3);
			\draw [style=arrow, bend left, draw=coolred!150] (2) to (0);
			\draw [style=arrow, bend right, draw=coolgreen!120] (2) to (0);
			\draw [style=arrow, bend right, draw=coolblue!120] (3) to (0);
			\draw [style=arrow, bend left] (3) to (0);
			\draw [style=arrow, bend left, looseness=0.75, draw=coolred!150] (6) to (7);
			\draw [style=arrow, draw=coolblue!120] (7) to (6);
			\draw [style=arrow, draw=coolred!150] (4) to (3);
			\draw [style=arrow, bend right=60, draw=coolblue!120] (4) to (3);
			\draw [style=arrow, bend left=60] (4) to (3);
			\draw [style=arrow, bend right,, draw=coolpink!140] (3) to (4);
			\draw [style=arrow, bend left, draw=coolorange!115] (3) to (4);
			\draw [style=arrow, draw=coolblue!120] (5) to (2);
			\draw [style=arrow, bend right=60] (5) to (2);
			\draw [style=arrow, bend left=60, draw=coolred!150] (5) to (2);
			\draw [style=arrow, bend right, draw=coolgreen!120] (2) to (5);
			\draw [style=arrow, bend left, draw=coolpink!140] (2) to (5);
			\draw [style=arrow, bend right=45, looseness=0.75, , draw=coolred!150] (5) to (4);
			\draw [style=arrow, bend left=75, looseness=0.75, , draw=coolblue!120] (4) to (5);
			\draw [style=arrow, bend right, color=coolred!150] (6) to (5);
			\draw [style=arrow, bend left, draw=coolblue!120] (7) to (4);
			\draw [style=arrow, bend right=45, draw=coolorange!115] (4) to (7);
			\draw [style=arrow, bend right=15] (4) to (7);
			\draw [style=arrow, bend left=15] (5) to (6);
			\draw [style=arrow, bend left=45,draw=coolgreen!120] (5) to (6);
		\end{pgfonlayer}
	\end{tikzpicture}
	\caption{A matroid-based packing of  5 arborescences.}
	\label{fig:mbased-packing-example-packing}
	\end{subfigure}
	\caption{Instance of \Cref{bboboibo}.}
	\label{fig:mbased-packing-example}
\end{figure}

We provide an  instance of \Cref{bboboibo} in \Cref{fig:mbased-packing-example} with the given digraph $D=(V,A)$ and the matroid ${\sf M}=(S,\mathcal{I}_{\sf M})$ where $S=\{s_1,s'_1,s_2,s_3,s_4\}$ and a set is in $\mathcal{I}_{\sf M}$ if and only if it contains at most one of $s_1$ and $s'_1$. In \Cref{fig:mbased-packing-example-bisets}, we give an OW laminar biset family of subatoms such that, for the function ${\sf f}({\sf X}) = r_{\sf M}(S) - r_{\sf M}(S_{X_W}) - d^-_A(X_O)$, we have \mbox{${\sf f}({\bf \color{coolred!150}{\sf A}})=2, {\sf f}({\bf \color{coolpink!125}{\sf B}}) = 1, {\sf f}({\bf \color{coolorange!150}{\sf C}}) = 1,$ and ${\sf f}({\bf \color{coolblue!100}{\sf D}}) = 1$}. Thus, by \Cref{bboboibo}, the given digraph requires at least $5$ arborescences and a matroid-based packing of 5 arborescences is  shown in \Cref{fig:mbased-packing-example-packing}.

Theorem \ref{bboboibo} will be later obtained from Theorem \ref{bboboiboreach}.

\medskip

We now provide the answer to the decomposition problem for the matroid-based version. 
It will be obtained from Theorem \ref{bboboibo}.

\begin{thm} \label{vouyfy}
Let $D=(V,A)$ be a  digraph, $S$ a multiset of vertices in $V$, and ${\sf M}=(S,r_{\sf M})$ a matroid. There exists a decomposition of $A$ into an ${\sf M}$-based packing of arborescences  in $D$ if and only if \eqref{kjbvgc} holds and for every OW laminar biset family ${\cal P}$ of subatoms, 
\begin{eqnarray} \label{ljhfygdtu}
	\sum_{{\sf X}\in\mathcal{P}}(r_{\sf M}(S)-r_{\sf M}(S_{X_W})-d_A^-(X_O)) 	&	\le	&r_{\sf M}(S)|V|-|A|. 
\end{eqnarray}
\end{thm}

We conclude  by showing that Theorem \ref{bboboibo} implies all the results of  Subsection \ref{mbpoa}.

\begin{cl}
Theorem \ref{bboboibo} implies Theorem \ref{thmddgnsz}.
\end{cl}

 \begin{proof} Let $(D,S,{\sf M})$ be an instance of Theorem \ref{thmddgnsz} satisfying \eqref{matcondori1} and \eqref{ddgnszcond}. Let $\ell=\ell'=|S|.$ The condition \eqref{ellell} trivially holds.
By \eqref{matcondori1}, we have $\sum_{v\in V}r_{\sf M}(S_v)=\sum_{v\in V}|S_v|=|S|,$ so \eqref{jbuoouou} is satisfied. By \eqref{ddgnszcond}, \eqref{kjbvgc} holds.
 Further, let ${\cal P}$ be an OW laminar  biset family  of subatoms. Then the inner sets of the bisets in ${\cal P}$ are disjoint. By \eqref{ddgnszcond}, the submodularity and the subcardinality of $r_{\sf M},$ and since $X_I$'s are disjoint, we have  $$\sum_{{\sf X}\in\mathcal{P}}(r_{\sf M}(S)-r_{\sf M}(S_{X_W})-d_A^-(X_O))\le \sum_{{\sf X}\in\mathcal{P}}(r_{\sf M}(S_{X_O})-r_{\sf M}(S_{X_W}))\le \sum_{{\sf X}\in\mathcal{P}}r_{\sf M}(S_{X_I})\le \sum_{{\sf X}\in\mathcal{P}}|S_{X_I}|\le |S|,$$ so \eqref{dknjkbduzv} is satisfied. Hence, by Theorem \ref{bboboibo}, there exists an ${\sf M}$-based packing of $s$-arborescences $(s\in S)$ in $D$, that is, a complete ${\sf M}$-based packing of arborescences in $D.$
 \end{proof}

\begin{cl}
Theorem \ref{bboboibo} implies Theorem \ref{mbpsaori}.
\end{cl}

 \begin{proof}
 Let $(D,S,{\sf M})$ be an instance of Theorem \ref{mbpsaori} satisfying \eqref{mbpsaoricond}. 
 Let $\ell=\ell'=r_{{\sf M}}(S)$. The condition \eqref{ellell} trivially holds.
 Then, by the submodularity  of $r_{\sf M},$ we have $\sum_{v\in V}r_{\sf M}(S_v)\ge r_{\sf M}(\bigcup_{v\in V}S_v)=r_{\sf M}(S),$ so \eqref{jbuoouou} is satisfied. Applying \eqref{mbpsaoricond} for every subatom as a subpartition we get that \eqref{kjbvgc} holds.
 To show that \eqref{dknjkbduzv} also holds, let  ${\cal P}$ be an OW laminar  biset family of subatoms.  
\noindent  For all ${\sf Y}\in\mathcal{P},$ let 
\begin{eqnarray*}
\text{{\boldmath$\mathcal{P}_{\sf Y}$}} &=&\{{\sf Z}\in\mathcal{P}-{\sf Y}: Z_O\subseteq Y_W\},\\
\text{{\boldmath$\mathcal{Q}_{\sf Y}$}} &=&\{{\sf Z}\in\mathcal{P}_{\sf Y}: \text{ there exists no } {\sf Z}'\in\mathcal{P}_{\sf Y}-{\sf Z}\text{ such that } Z_O\subseteq Z'_W\}.
\end{eqnarray*}
Note  that $\mathcal{P}_{\sf Y}=\bigcup_{Z\in \mathcal{Q}_{\sf Y}}(Z\cup\mathcal{P}_{\sf Z})$ and the outer sets of the bisets in $\mathcal{Q}_{\sf Y}$ are mutually disjoint. Indeed, if ${\sf Z}^1,{\sf Z}^2\in\mathcal{Q}_{\sf Y}$ and $Z^1_O\cap Z^2_O\neq\emptyset$, then, since ${\cal P}$ is OW laminar, we have that $Z^1_O\subseteq Z^2_W$ or $Z^2_O\subseteq Z^1_W$, so either $Z^1\notin\mathcal{Q}_{\sf Y}$ or $Z^2\notin\mathcal{Q}_{\sf Y}$ which is a contradiction.
\medskip

\noindent Let us introduce the following biset function:  
{\boldmath${\sf f}$}$({\sf X})=r_{\sf M}(S)-r_{\sf M}(S_{X_W})-d_A^-(X_O)$ for all ${\sf X}\in\mathcal{P}.$
We now  claim that we have 
\begin{eqnarray}\label{hhjc}
\sum_{{\sf Z}\in\mathcal{P}_{\sf Y}}{\sf f}({\sf Z})\le r_{\sf M}(S_{Y_W}) \text{  for every }{\sf Y}\in\mathcal{P}.
\end{eqnarray}
 If $\mathcal{P}_{\sf Y}=\emptyset,$ then, by the non-negativity of $r_{\sf M},$ the \eqref{hhjc} holds. Suppose that \eqref{hhjc} holds for every ${\sf Z}\in\mathcal{P}_{\sf Y}$. We show that it also holds for ${\sf Y}.$ By the hypothesis, the definition of ${\sf f}$, since the outer sets of the bisets in $\mathcal{Q}_{\sf Y}$ are mutually disjoint, by \eqref{mbpsaoricond} applied for these sets, since ${\cal P}$ is OW laminar and by the monotonicity of $r_{\sf M},$ we have \eqref{hhjc} for ${\sf Y}:$
\begin{eqnarray*}
\sum_{{\sf Z}\in\mathcal{P}_{\sf Y}}{\sf f}({\sf Z}) & = & \sum_{{\sf Z}\in\mathcal{Q}_{\sf Y}}({\sf f}({\sf Z})+\sum_{{\sf X}\in\mathcal{P}_{\sf Z}}{\sf f}({\sf X}))\le\sum_{{\sf Z}\in\mathcal{Q}_{\sf Y}}({\sf f}({\sf Z})+r_{\sf M}(S_{Z_W})) \\ & = & \sum_{{\sf Z}\in\mathcal{Q}_{\sf Y}}(r_{\sf M}(S)-d_A^-(Z_O))\le r_{\sf M}(S_{\bigcup_{{\sf Z}\in\mathcal{Q}_{\sf Y}}Z_O})\le r_{\sf M}(S_{Y_W}).
\end{eqnarray*}
Let ${\cal P}'={\cal P}\cup \{{\sf V}\}$ where ${\sf V}=(V,\emptyset)$. Then ${\cal P}'$ is also an OW laminar  biset family of $V$. Since the above arguments also work for ${\cal P}'$, we have, by \eqref{hhjc} for ${\sf V},$ $\sum_{{\sf X}\in\mathcal{P}}{\sf f}({\sf X})=\sum_{{\sf X}\in\mathcal{P}'_{{\sf V}}}{\sf f}({\sf X})\le r_{\sf M}(S_{V})=r_{\sf M}(S)$, so \eqref{dknjkbduzv} holds.
Hence, by Theorem \ref{bboboibo}, there exists an ${\sf M}$-based packing of $s$-arborescences $(s\in S^*\subseteq S)$ in $D$ with $|S^*|=r_{\sf M}(S)$, and hence each arborescence in the packing is spanning.
 \end{proof}

\begin{cl}
Theorem \ref{bboboibo} implies Theorem \ref{vouyfy}.
\end{cl}

  \begin{proof}
Let $(D=(V,A),S,{\sf M})$ be an instance of Theorem \ref{vouyfy}. To see the {\bf necessity}, suppose that there exists a decomposition of $A$ into an ${\sf M}$-based packing of arborescences  in $D$ with root set $S^*.$ Then, by Theorem \ref{thmddgnsz}, \eqref{ddgnszcond} and hence \eqref{kjbvgc} holds. Since $|S^*|=\sum_{v\in V}|S^*_v|=\sum_{v\in V}(r_{\sf M}(S)-d_A^-(v))=r_{\sf M}(S)|V|-|A|$, by  Theorem \ref{bboboibo},
 \eqref{dknjkbduzv} holds for $\ell'=r_{\sf M}(S)|V|-|A|$ and hence \eqref{ljhfygdtu} holds.

To see the {\bf sufficiency}, suppose that \eqref{kjbvgc} and \eqref{ljhfygdtu} hold. Let $\ell=\ell'=r_{\sf M}(S)|V|-|A|.$ Then  conditions \eqref{ellell}, \eqref{kjbvgc} and \eqref {dknjkbduzv}  hold. By \eqref{kjbvgc} applied for all $v\in V,$ we get that $\sum_{v\in V}r_{\sf M}(S_v)\ge\sum_{v\in V}(r_{\sf M}(S)-d_A^-(v))=r_{\sf M}(S)|V|-|A|,$ so \eqref{jbuoouou} also holds. Hence, by Theorem \ref{bboboibo}, there exists an ${\sf M}$-based  packing of  arborescences  in $D$ with arc set $B$ and root set $S^*$ such that $|S^*|=r_{\sf M}(S)|V|-|A|$. Since  $|B|=\sum_{v\in V}d_B^-(v)=\sum_{v\in V}(r_{\sf M}(S)-|S^*_v|)=r_{\sf M}(S)|V|-|S^*|=|A|$, we have in fact a decomposition of $A,$ and the proof of Theorem \ref{vouyfy} is complete.
 \end{proof}

\subsection{Matroid-reachability-based packing of arborescences}\label{mrbpoa}

In this section we present results on matroid-reachability-based packing of arborescences. 
\medskip

We need the following two functions which are non-zero only on  petals and petal bisets.
Let $D=(V,A)$ be a  digraph, $S$ a multiset of vertices in $V$ and ${\sf M}$ a matroid on $S$ with rank function $r_{\sf M}$. Recall that $\hat{\mathcal{Z}}$ is the set of petals and we have $P_{C_Z}=P_Z$ for all $Z\in\hat{\mathcal{Z}}$; $\hat{\mathcal{Z}}_{\sf b}$ is the set of petal bisets  and we have   $P_{C_{\sf X}}=P_{X_O}$  for all ${\sf X}\in\hat{\mathcal{Z}}_{\sf b}.$
%
Let the set function {\boldmath$\hat{p}$} and the biset function {\boldmath$\hat{\sf p}$} on $V$ be defined as follows: 
	\[
		\hat{p}(Z)= 
		\begin{cases}
			r_{{\sf M}}(S_{P_{C_Z}})-r_{{\sf M}}(S_Z)	&	Z\in \hat{\mathcal{Z}},\\
			0			&	\text{otherwise,}
		\end{cases}
\hskip 1truecm 		\hat{\sf p}({\sf X})= 
		\begin{cases}
			r_{{\sf M}}(S_{P_{C_{\sf X}}})-r_{{\sf M}}(S_{X_O})	&	{\sf X}\in \hat{\mathcal{Z}}_{\sf b},\\
			0			&	\text{otherwise.}
		\end{cases}
	\] 
%
	
The following properties of these functions will be crucial.
	
\begin{cl}\label{psupmod} The following hold.

 (a) $\hat p$  is a  supermodular function on  core-intersecting sets of $\hat{\mathcal{Z}}$.
 
 (b) $\hat{\sf p}$ is a positively intersecting supermodular biset function on $V.$
\end{cl}

\begin{proof}
(a) Let $Z$ and $Z'$ be  core-intersecting sets in $\hat{\mathcal{Z}}$.  It follows that there exists an atom $C$ of $D$ such that 
\begin{eqnarray}
\hat{p}(Z)&=r_{{\sf M}}(S_{P_C})-r_{{\sf M}}(S_{Z}),&\label{supmod1}\\
\hat{p}(Z')&=r_{{\sf M}}(S_{P_C})-r_{{\sf M}}(S_{Z'}),&\label{supmod2}\\
\emptyset\neq Z\cap C, & Z\subseteq P_C, &d_A^-(Z-C)=0,\label{supmod5} \\
\emptyset\neq Z'\cap C, & Z'\subseteq P_C, &d_A^-(Z'-C)=0. \label{supmod6}
\end{eqnarray}
Then, by \eqref{supmod5} and \eqref{supmod6}, we have 
\begin{eqnarray}
\emptyset\neq (Z\cap Z')\cap C,& Z\cap Z'\subseteq P_C, &d_A^-((Z\cap Z')-C)=0,\label{supmod7}\\ 
\emptyset\neq (Z\cup Z')\cap C,& Z\cup Z'\subseteq P_C, &d_A^-((Z\cup Z')-C)=0.\label{supmod8}
\end{eqnarray}
By \eqref{supmod7} and \eqref{supmod8}, we have $Z\cap Z', Z\cup Z'\in \hat{\mathcal{Z}}$, so
\begin{eqnarray}
\hat{p}(Z\cap Z')&=r_{{\sf M}}(S_{P_C})-r_{{\sf M}}(S_{Z\cap Z'}),&\label{supmod3}\\
\hat{p}(Z\cup Z')&=r_{{\sf M}}(S_{P_C})-r_{{\sf M}}(S_{Z\cup Z'}).&\label{supmod4}
\end{eqnarray}
Since $r_{{\sf M}}(S_Z)+r_{{\sf M}}(S_{Z'})\ge r_{{\sf M}}(S_{Z\cap Z'})+r_{{\sf M}}(S_{Z\cup Z'})
$, we get, by \eqref{supmod1},\eqref{supmod2},\eqref{supmod3},\eqref{supmod4}, that  $\hat{p}(Z)+\hat{p}(Z')\le \hat{p}(Z\cap Z')+\hat{p}(Z\cup Z'),$ and (a) follows.
\medskip

(b) The same proof as in (a) also works for (b).
\end{proof}

A common generalization of Theorems \ref{reach1} and \ref{thmddgnsz} was given by Kir\'aly \cite{cskir} where he characterized the existence of a complete matroid-reachability-based packing of arborescences.

\begin{thm}[Kir\'aly \cite{cskir}] \label{thmCsaba}
Let $D=(V,A)$ be a  digraph, $S$ a multiset of vertices in $V$, and ${\sf M}=(S,\mathcal{I}_{\sf M})$ a matroid with rank function $r_{\sf M}$. 
There exists a complete {\sf M}-reachability-based packing of arborescences in $D$  if and only if   \eqref{matcondori1} holds and 
	\begin{eqnarray}
		d^-_{A}(Z)\geq r_{{\sf M}}(S_{P_Z})-r_{{\sf M}}(S_Z)&& \text{ for every } Z\subseteq V.\label{csabicond}
	\end{eqnarray}
\end{thm}

For the free matroid, Theorem \ref{thmCsaba} reduces to Theorem \ref{reach1}.
If $r_{\sf M}(S_{P_v})=r_{\sf M}(S)$ for all $v\in V,$ then  Theorem \ref{thmCsaba} reduces to Theorem \ref{thmddgnsz}.
\medskip

Another characterization of the existence of a  matroid-reachability-based packing of arborescences was given by Gao and Yang \cite{GY}. 

\begin{thm}[Gao, Yang \cite{GY}] \label{thmGY}
Let $D=(V,A)$ be a  digraph, $S$ a multiset of vertices in $V$, and ${\sf M}=(S,r_{\sf M})$ a matroid. 

(a) There exists an {\sf M}-reachability-based packing of arborescences in $D$  if and only if   
	\begin{eqnarray}\label{GYcond}
\hskip .95truecm\	{\sf d}_A^-({\sf X}) 	&\ge	&	\hat{\sf p}({\sf X}) \hskip .44truecm\text{ for every  biset {\sf X} on $V$,}  
	\end{eqnarray}
or equivalently
	\begin{eqnarray}\label{GYcond2}
		d_A^-(Z) 	&	\ge	&	\hat{p}({Z}) \hskip .44truecm\text{ for every  $Z\in\hat{\mathcal{Z}}$.}  
	\end{eqnarray}
	
(b) There exists a complete {\sf M}-reachability-based packing of arborescences in $D$  if and only if  \eqref{matcondori1} and \eqref{GYcond2}  hold.
\end{thm}

In \cite{pmh} it was proved that Theorems  \ref{thmCsaba} and \ref{thmGY} are equivalent.

\subsubsection{New results on matroid-reachability-based packing of arborescences}

In this subsection we provide our main results. 
\medskip

We start with the following polyhedral result.
\begin{thm}\label{mainTDI}
Let $D=(V,A)$ be a  digraph, $S$ a multiset of vertices in $V$, and ${\sf M}=(S,r_{\sf M})$ a matroid such that \eqref{GYcond2} holds.  Then the   system defined by \eqref{tdicond3} and \eqref{tdicondFJ2} is TDI.
\begin{eqnarray}
	x(\delta_A^-(Z))&\ge& \hat p(Z) \hskip 1truecm\text{ for every } Z\in\hat{\mathcal{Z}},\label{tdicond3} \\
	\mathbb{1}	\hskip .25truecm 	\ge 	\hskip .25truecm	x &\ge & \mathbb{0}.\label{tdicondFJ2}
\end{eqnarray}
\end{thm}

\begin{proof} 
It is clear that the polyhedron defined by \eqref{tdicond3} and \eqref{tdicondFJ2} coincides with the polyhedron defined by \eqref{tdicondFJ2} and
\begin{eqnarray}
	x(\delta_A^-({\sf X}))	&	\ge	&	\hat{\sf p}({\sf X}) \hskip 1truecm\text{ for every  biset {\sf X} on $V$.}\label{tdicond1} 
\end{eqnarray}
Claim \ref{psupmod}(b), \eqref{GYcond} and Theorem \ref{FJTDI} immediately imply that the system defined by \eqref{tdicondFJ2} and \eqref{tdicond1} is TDI. Let {\sf X} be a biset on $V$. If ${\sf X}\notin\hat{\mathcal{Z}}_{\sf b}$, then $\hat{\sf p}({\sf X})=0,$ so the inequality $x(\delta_A^-({\sf X}))\ge \hat{\sf p}({\sf X})=0$ is the sum of the inequalities $x(a)\ge 0$ for all $a\in\delta_A^-({\sf X}).$ Otherwise, ${\sf X}\in\hat{\mathcal{Z}}_{\sf b}$ so $\hat{\sf p}({\sf X})=r_{{\sf M}}(S_{P_C})-r_{{\sf M}}(S_{X_O})$ for an atom $C$ of $D$. Then $Z=X_O\in\hat{\mathcal{Z}},$ hence \eqref{tdicond1} for ${\sf X}$ and \eqref{tdicond3} for $Z$ coincide.  Then, by Theorem  \ref{tdiclaim}, the system defined by \eqref{tdicond3} and \eqref{tdicondFJ2} is also  TDI.
\end{proof}

In order to characterize the existence of a matroid-reachability-based $(\ell,\ell')$-limited packing of  arborescences, our strategy is to minimize the number of roots of the arborescences in the packing. To achieve this we consider the extended version of the problem where the elements of the matroid correspond to different vertices of the extended graph. For an instance $(D=(V,A), S, \ell, \ell', {\sf M}=(S,r_{\sf M}))$  of the problem, let {\boldmath$D'$} $=(V\cup S',A\cup A')$ be obtained from $D$ by adding a new vertex set {\boldmath$S'$} containing one vertex $s'$ for every $s\in S$ and adding a new arc set {\boldmath$A'$} containing one arc $s's$ for every $s\in S$. Let {\boldmath${\sf M}'$} be a copy of {\sf M} on $S'.$ We say that a family $\mathcal{Z}$ of subsets of $V\cup S$ is {\it $A'$-disjoint} if every arc of $A'$ enters at most one member $\mathcal{Z}$.
\medskip

We are  ready to present the following intermediary result, for the proof see Subsection \ref{pr11}.

\begin{thm} \label{bboboiboreach1}
Let $D=(V\cup S,A^*)$ be a digraph ($A'$ being the set of arcs leaving $S$ and $A=A^*-A'$) such that no arc enters $s$ and exactly one arc leaves $s$ for every vertex $s$ of $S,$ $\ell,\ell'\in\mathbb{Z}_+$, and ${\sf M}=(S,r_{\sf M})$ a matroid. There exists an ${\sf M}$-reachability-based packing of  arborescences in $D$ using at least $\ell$ and at most $\ell'$ arcs of $A'$ if and only if \eqref{ellell}  holds and
\begin{eqnarray} 
\sum_{v\in V}r_{\sf M}(N_{A'}^-(v))	&	\ge	&	\ell,	\label{jzbdjzbzb}	\\
	r_{\sf M}(S\cap P_Z)-r_{\sf M}(S\cap Z) &	\le	&d_{A^*}^-(Z) \hskip .5truecm  \text{for every   $Z\in\hat{\mathcal{Z}},$}\hskip .2truecm \label{1dknjkbduzvreachex}\\
	\sum_{Z\in\mathcal{Z}}(r_{\sf M}(S\cap P_Z)-r_{\sf M}(S\cap Z)-d_A^-(Z)) 	&\le	&\ell' \hskip .5truecm \forall \text{$A'$-disjoint core-laminar subset $\mathcal{Z}$ of $\hat{\mathcal{Z}}.$}\hskip .6truecm \label{1dknjkbduzvreach}
\end{eqnarray}
\end{thm}

Recall that $\mathcal{X}$ is the set of generalized petal bisets.
We now present our main result. It will be obtained from Theorem \ref{bboboiboreach1}, for the proof see Subsection \ref{pr12}.

\begin{thm} \label{bboboiboreach}
Let $D=(V,A)$ be a  digraph, $S$ a multiset of vertices in $V$, $\ell,\ell'\in\mathbb{Z}_+$, and ${\sf M}=(S,r_{\sf M})$ a matroid. There exists an ${\sf M}$-reachability-based $(\ell,\ell')$-limited packing of  arborescences in $D$ if and only if \eqref{ellell}, \eqref{jbuoouou} and \eqref{GYcond} hold and 
\begin{eqnarray} 
	\sum_{{\sf X}\in\mathcal{P}}(r_{\sf M}(S_{P_{X_I}})-r_{\sf M}(S_{X_W})-d_A^-(X_O))& \le	&\ell'		\hskip .4truecm  \text{ for every OW laminar biset family } {\cal P} \text{ of }  \mathcal{X}.\hskip .8truecm	\label{dknjkbduzvreach}
\end{eqnarray}
\end{thm}

We finally provide the answer for the decomposition problem for the matroid-reachability-based version. It will be easily obtained from the previous result.

\begin{thm} \label{vouyfy2}
Let $D=(V,A)$ be a  digraph, $S$ a multiset of vertices in $V$, and ${\sf M}=(S,r_{\sf M})$ a matroid. There exists a decomposition of $A$ into an ${\sf M}$-reachability-based packing of arborescences  in $D$ if and only if \eqref{GYcond2} holds and for every OW laminar biset family ${\cal P}$ of $\mathcal{X},$ 
\begin{eqnarray} \label{ljhfygdtu2}
\sum_{{\sf X}\in\mathcal{P}}(r_{\sf M}(S_{P_{X_I}})-r_{\sf M}(S_{X_W})-d_A^-(X_O))& \le	&(\sum_{v\in V}r_{\sf M}(S_{P_v}))-|A|.	
\end{eqnarray}
\end{thm}

We conclude by showing that Theorem \ref{bboboiboreach} implies all the results of  Subsection \ref{mrbpoa}.

\begin{cl}
Theorem \ref{bboboiboreach} implies Theorem \ref{bboboibo}.
\end{cl}

\begin{proof}
 Let $(D=(V,A), {\sf M}=(S,r_{\sf M}), \ell,\ell')$ be an instance of Theorem \ref{bboboibo} satisfying \eqref{kjbvgc}, \eqref{ellell}, \eqref{jbuoouou},  and \eqref{dknjkbduzv}. We show that \eqref{dknjkbduzv} implies \eqref{GYcond} and \eqref{dknjkbduzvreach}. 
 First we mention that  
\begin{eqnarray} 
	r_{\sf M}(S) &\le& r_{\sf M}(S_{Z})\hskip .9truecm \text{ for every non-empty } Z\subseteq V \text{ with }  d_A^-(Z)=0.\label{igcgch}
 \end{eqnarray}
Indeed, by $Z\neq\emptyset$ and $d_A^-(Z)=0,$ there exists a smallest non-empty  $Y\subseteq Z$ such that $d_A^-(Y)=0.$ Then $Y$ is an atom of $D.$ Thus, by \eqref{kjbvgc} and since $r_{\sf M}$ is monotone, we get that \eqref{igcgch} also holds.
 \medskip
 
 To show  \eqref{GYcond}, let  ${\sf X}\in\hat{\mathcal{Z}}_{\sf b}.$ If $X_W\neq\emptyset,$ then, by the monotonicity of $r_{\sf M}$, $d_A^-(X_W)=0$ and  \eqref{igcgch}, we get that $r_{\sf M}(S_{P_{X_I}})-r_{\sf M}(S_{X_O})\le r_{\sf M}(S)-r_{\sf M}(S_{X_W})\le 0\le d_A^-(X_O)$, and \eqref{GYcond} holds. If $X_W=\emptyset,$ then, $X_O$ is an subatom, so, by the monotonicity of $r_{\sf M}$ and \eqref{kjbvgc}, we get that $r_{\sf M}(S_{P_{X_I}})-r_{\sf M}(S_{X_O})\le r_{\sf M}(S)-r_{\sf M}(S_{X_O})\le d_A^-(X_O),$ and \eqref{GYcond} holds.
  \medskip

To show that \eqref{dknjkbduzvreach} also holds let ${\cal P}$ be an OW laminar biset family of $\mathcal{X}.$  For every ${\sf X}\in\mathcal{P}$, we may suppose without loss of generality, by the monotonicity of $r_{\sf M}$,   that we have $1\le r_{\sf M}(S_{P_{X_I}})-r_{{\sf M}}(S_{X_W})-d_A^-(X_O)\le r_{{\sf M}}(S)-r_{{\sf M}}(S_{X_W-C_{\sf X}}).$ Since $d_A^-(X_W-C_{\sf X})=0,$ it follows, by \eqref{igcgch}, that $X_W-C_{\sf X}=\emptyset$, that is, $X_O\subseteq C_{\sf X}$, so $X_O$ is a subatom. Then, since $r_{\sf M}$ is monotone and by \eqref{dknjkbduzv}, we have $\sum_{{\sf X}\in\mathcal{P}}(r_{\sf M}(S_{P_{X_I}})-r_{\sf M}(S_{X_W})-d_A^-(X_O))\le \sum_{{\sf X}\in\mathcal{P}}(r_{\sf M}(S)-r_{\sf M}(S_{X_W})-d_A^-(X_O))\le \ell',$ so  \eqref{dknjkbduzvreach} holds.
 \medskip

By Theorem \ref{bboboiboreach}, there exists an ${\sf M}$-reachability-based $(\ell,\ell')$-limited packing of  arborescences in $D$. Since, by \eqref{igcgch}, we have $r_{\sf M}(S_{P_{v}})\ge r_{\sf M}(S)$ for every $v\in V,$ the packing is ${\sf M}$-based and the proof of Theorem \ref{bboboibo} is complete.
\end{proof}

\begin{cl}
Theorem \ref{bboboiboreach} implies Theorem \ref{thmGY}.
\end{cl}

\begin{proof}
 Let $(D=(V,A), {\sf M}=(S,r_{\sf M}))$ be an instance of Theorem \ref{thmGY} satisfying \eqref{GYcond}. 
 Let $\ell=0$ and $\ell'=|S|.$ Then $(D=(V,A), {\sf M}=(S,r_{\sf M}), \ell,\ell')$ is an instance of Theorem \ref{bboboiboreach}.
 We now show that all the conditions of Theorem \ref{bboboiboreach} hold. Conditions \eqref{ellell} and \eqref{jbuoouou} trivially hold and  \eqref{GYcond} holds by assumption.
 To show that  \eqref{dknjkbduzvreach} also holds let ${\cal P}$ be an OW laminar biset family  of $\mathcal{X}$. By \eqref{GYcond}, the submodularity and the subcardinality of $r_{\sf M},$ and since $X_I$'s are disjoint, we have 
 \begin{eqnarray*} 
	0	&	\ge	&	\sum_{{\sf X}\in\mathcal{P}}(r_{\sf M}(S_{P_{X_O}})-r_{\sf M}(S_{X_O})-d_A^-(X_O))			\\
		&	\ge 	&	\sum_{{\sf X}\in\mathcal{P}}(r_{\sf M}(S_{P_{X_I}})-r_{\sf M}(S_{X_W})-|S_{X_I}|-d_A^-(X_O))	\\			&	\ge	&	\sum_{{\sf X}\in\mathcal{P}}(r_{\sf M}(S_{P_{X_I}})-r_{\sf M}(S_{X_W})-d_A^-(X_O))-|S|,
\end{eqnarray*}
so \eqref{dknjkbduzvreach}  holds. Then, by Theorem \ref{bboboiboreach}, there exists an ${\sf M}$-reachability-based packing of  arborescences in $D$ (containing at most $|S|$ arborescences) and the proof of Theorem \ref{thmGY} is completed.
\end{proof}

\begin{cl}
Theorem \ref{bboboiboreach} implies Theorem \ref{vouyfy2}.
\end{cl}

 \begin{proof}
Let $(D=(V,A),S,{\sf M})$ be an instance of Theorem \ref{vouyfy2}. To see the {\bf necessity}, suppose that there exists a decomposition of $A$ into an ${\sf M}$-reachability-based packing of arborescences  in $D$ with root set $S^*.$ Then, by Theorem \ref{thmGY}, \eqref{GYcond2}  holds. Since $|S^*|=\sum_{v\in V}|S^*_v|=\sum_{v\in V}(r_{\sf M}(S_{P_v})-d_A^-(v))=\sum_{v\in V}r_{\sf M}(S_{P_v})-|A|$, by  Theorem \ref{bboboiboreach},
 \eqref{dknjkbduzvreach} holds for $\ell'=\sum_{v\in V}r_{\sf M}(S_{P_v})-|A|$ and hence \eqref{ljhfygdtu2} holds.

To see the {\bf sufficiency}, suppose that \eqref{GYcond2} and \eqref{ljhfygdtu2} hold. Let $\ell=\ell'=\sum_{v\in V}r_{\sf M}(S_{P_v})-|A|.$ Then   \eqref{ellell}, \eqref{GYcond2} (and hence \eqref{GYcond}) and \eqref {dknjkbduzvreach}  hold. By \eqref{GYcond2} applied for all $v\in V,$ we get that $\sum_{v\in V}r_{\sf M}(S_v)\ge\sum_{v\in V}(r_{\sf M}(S_{P_v})-d_A^-(v))=\sum_{v\in V}r_{\sf M}(S_{P_v})-|A|,$ so \eqref{jbuoouou} also holds. Hence, by Theorem \ref{bboboiboreach}, there exists an ${\sf M}$-reachability-based  packing of  arborescences  in $D$ with arc set $B$ and root set $S^*$ such that $|S^*|=\sum_{v\in V}(r_{\sf M}(S_{P_v})-|A|$. Since  $|B|=\sum_{v\in V}d_B^-(v)=\sum_{v\in V}(r_{\sf M}(S_{P_v})-|S^*_v|)=\sum_{v\in V}r_{\sf M}(S_{P_v})-|S^*|=|A|$, we have in fact a decomposition of $A,$ and the proof of Theorem \ref{vouyfy} is complete.
 \end{proof}

\subsection{Packing of branchings}

We complete the section on packings in digraphs by some results on packing branchings that can be derived from our results.
\medskip

Let  $D=(V,A)$ be a digraph, ${\cal S}$ a family of subsets of $V$, and  ${\sf M}=(\mathcal{S},r_{\sf M})$ a matroid.
Let {\boldmath$\hat S$} $=\bigcup_{S\in {\cal S}}S$ where the union is taken by multiplicities, so $\hat S$ is a multiset of $V$ and ${\cal S}$ is a partition of $\hat S.$ Recall that {\boldmath${\sf M}_{\cal S}^1$} is the partition matroid on $\hat S$ with value $1$ on each $S\in {\cal S}.$ Let {\boldmath$\hat{{\sf p}}_{\mathcal{S}}({\sf X})$} $=r_{{\sf M}}(\mathcal{S}_{P_{C_{\sf X}}})-r_{{\sf M}}(\mathcal{S}_{X_O})$ if ${\sf X}\in \hat{\mathcal{Z}}_{\sf b}$ and $0$ otherwise.

\begin{thm}[Edmonds \cite{Egy}]\label{edmondsbranchings}
Let $D=(V,A)$ be a digraph and ${\cal S}$ a family of subsets of $V$.  There exists a packing of spanning $S$-branchings $(S\in {\cal S})$ in $D$ if and only if 
	\begin{eqnarray}\label{eed}
		|{\cal S}_X|+d^-_A(X)&\geq &|{\cal S}| \hskip .5truecm \text{ for every non-empty  $X\subseteq V.$}
	\end{eqnarray}
\end{thm}

For ${\cal S}=\{\{s\}: s\in S\},$ Theorem \ref{edmondsbranchings} reduces to Theorem \ref{edmondsarborescencesmulti}.
Theorem \ref{edmondsbranchings} can be obtained from Theorem \ref{thmddgnsz} for the  matroid ${\sf M}_{\cal S}^1$.
We note that to have  a packing of spanning $S$-branchings $(S\in {\cal S})$ in $D$ it is sufficient that \eqref{eed} holds for every subatom.
\medskip

We  provide a  generalization of Theorem  \ref{edmondsbranchings} with bounds on the total number of roots.

\begin{thm}\label{bfurccsuj}
Let $D=(V,A)$ be a digraph, ${\cal S}$ a family of subsets of $V$, and $\ell, \ell'\in\mathbb Z_+$. There exists  an $(\ell,\ell')$-limited packing of spanning $S'$-branchings $(\emptyset\neq S'\subseteq S\in {\cal S})$ in $D$  if and only \eqref{ellell} holds and 
 \begin{eqnarray}
 	\ell	&	\le	&	\sum_{S\in {\cal S}}|{S}|, \label{lkjejhcuheu2}\\
	|{\cal S}_X|+d^-_A(X)&\geq &|{\cal S}| \hskip .2truecm \text{ for every subatom } X \text{ of } D,\label{lkkjbkjvjh}\\
	\sum_{{\sf X}\in {\cal P}}(|{\cal S}|-|{\cal S}_{X_W}|-d_A^-(X_O))		&	\le	&	\ell'
	\hskip .4truecm  \text{ for every OW laminar  biset family } {\cal P} \text{ of subatoms}.\hskip .4truecm \label{lkjejhcuheu}
\end{eqnarray}
\end{thm}


Theorem \ref{bfurccsuj} can be extended even to matroid-reachability-based packings as follows. 

\begin{thm} \label{bboboiboreachbr}
Let $D=(V,A)$ be a  digraph, $\mathcal{S}$ a family of subsets of $V$, $\ell,\ell'\in\mathbb{Z}_+$, and ${\sf M}=(\mathcal{S},r_{\sf M})$ a matroid. There exists an ${\sf M}$-reachability-based $(\ell,\ell')$-limited packing of  $S'$-branchings $(S'\subseteq S\in {\cal S})$ in $D$ if and only if \eqref{ellell} holds and 
\begin{eqnarray} 
		\sum_{v\in V}r_{\sf M}(\mathcal{S}_v)	&	\ge	&	\ell,	\label{jbuoououbr}	\\
	{\sf d}_A^-({\sf X}) 	&\ge	&	\hat{{\sf p}}_{\mathcal{S}}({\sf X}) \hskip .44truecm\text{ for every  biset {\sf X} on $V$,}  \label{iuvuoyct}\\
	\sum_{{\sf X}\in\mathcal{P}}(r_{\sf M}(\mathcal{S}_{P_{X_I}})-r_{\sf M}(\mathcal{S}_{X_W})-d_A^-(X_O))& \le	&\ell'		\hskip .4truecm  \text{ for every OW laminar biset family } {\cal P} \text{ of }  \mathcal{X}.\hskip .8truecm	\label{dknjkbduzvreachbr}
\end{eqnarray}
\end{thm}

Let us show some implications between these results.
\begin{cl}
Theorem \ref{bboboibo} implies Theorem \ref{bfurccsuj}.
\end{cl}

\begin{proof}
 Let $(D=(V,A), {\cal S}, \ell,\ell')$ be an instance of Theorem \ref{bfurccsuj} satisfying \eqref{ellell}, \eqref{lkjejhcuheu2}, \eqref{lkkjbkjvjh} and \eqref{lkjejhcuheu}. Then $(D, \hat S, \ell,\ell',{\sf M}_{\cal S}^1)$ is an instance of Theorem \ref{bboboibo}. We now show that it satisfies \eqref{kjbvgc}, \eqref{ellell}, \eqref{jbuoouou},  and \eqref{dknjkbduzv}. By assumption, \eqref{ellell} holds. By \eqref{lkjejhcuheu2},  we have $\sum_{v\in V}r_{{\sf M}_{\cal S}^1}(\hat S_v)=\sum_{v\in V}|\hat S_v|=\sum_{S\in {\cal S}}|{S}|\ge\ell,$ so \eqref{jbuoouou} holds for  $(D, \hat S,  \ell,\ell',{\sf M}_{\cal S}^1)$. Since  $r_{{\sf M}_{\cal S}^1}(\hat S_Z)=|{\cal S}_Z|$ for all $Z\subseteq V,$  \eqref{lkkjbkjvjh} implies that \eqref{kjbvgc}  holds and \eqref{lkjejhcuheu} implies that \eqref{dknjkbduzv}  holds for  $(D, \hat S,  \ell,\ell',{\sf M}_{\cal S}^1)$.  
Thus, by Theorem \ref{bboboibo}, there exists an ${{\sf M}_{\cal S}^1}$-based $(\ell,\ell')$-limited  packing of  $s$-arborescences $(s\in S^*\subseteq \hat S)$ in $D$. For every $S\in\mathcal{S},$ by the definition of ${\sf M}_{\cal S}^1$, the $s$-arborescences in the packing with $s\in S$ are vertex disjoint and hence form an $S'$-branching with $S'\subseteq S$. Since the packing is ${{\sf M}_{\cal S}^1}$-based and $r_{{\sf M}_{\cal S}^1}(\hat S)=|{\cal S}|,$ each $S'$-branching is spanning. We hence have  an $(\ell,\ell')$-limited packing of spanning $S'$-branchings $(\emptyset\neq S'\subseteq S\in {\cal S})$ in $D$. This completes the proof of Theorem \ref{bfurccsuj}.
\end{proof}

\begin{cl}
Theorem \ref{bfurccsuj} implies Theorem \ref{edmondsbranchings}.
\end{cl}

\begin{proof}
 Let  $(D,{\cal S})$ be an instance of Theorem \ref{edmondsbranchings} such  that \eqref{eed} holds. Let $\ell=\ell'=\sum_{S\in {\cal S}}|{S}|.$ Note that \eqref{ellell}, \eqref{lkjejhcuheu2}  trivially hold and, by \eqref{eed},  \eqref{lkkjbkjvjh} holds. Let ${\cal P}$ be  any OW laminar  biset family of subatoms. Then the inner sets of the bisets in ${\cal P}$ are disjoint. Thus, by \eqref{eed}, we have 
\begin{eqnarray*}
 \sum_{{\sf X}\in {\cal P}}(|{\cal S}|-|{\cal S}_{X_W}|-d_A^-(X_O))\le\sum_{{\sf X}\in {\cal P}}(|{\cal S}_{X_O}|-|{\cal S}_{X_W}|)\le\sum_{{\sf X}\in {\cal P}}|{\cal S}_{X_I}|\le \sum\limits_{\substack{{v\in \bigcup X_I}\\{\sf X}\in {\cal P}}}|{\cal S}_v|\le  \sum_{v\in V}|{\cal S}_v|=\sum_{S\in {\cal S}}|{S}|,
\end{eqnarray*}
  so \eqref{lkjejhcuheu} also holds. Hence, by Theorem  \ref{bfurccsuj}, there exists  a packing of spanning $S'$-branchings $(\emptyset\neq S'\subseteq S\in {\cal S})$ in $D$ such that $\sum_{S\in {\cal S}}|{S}'|=\sum_{S\in {\cal S}}|{S}|$, and Theorem \ref{edmondsbranchings} follows.
\end{proof}

\begin{cl}
Theorem \ref{bboboiboreach} implies Theorem \ref{bboboiboreachbr}.
\end{cl}

\begin{proof}
 Let $(D, {\sf M}=(\mathcal{S},r_{\sf M}), \ell,\ell')$ be an instance of Theorem \ref{bboboiboreachbr} satisfying \eqref{ellell}, \eqref{jbuoououbr},  \eqref{iuvuoyct}, and \ref{dknjkbduzvreachbr}. 
Let {\boldmath${\sf M}'_{\cal S}$} be the matroid on the multiset $\hat S=\bigcup_{S\in {\cal S}}S$ of vertices  obtained from ${\sf M}$ by replacing each $S\in {\cal S}$ by parallel elements on all $s\in S,$ that is $r_{ {\sf M}'_{\cal S}}(S')=r_{\sf M}(\mathcal{S}_{S'})$ for every $S'\subseteq \hat S.$ Then $(D, {\sf M}'_{\cal S}=(\hat S,r_{ {\sf M}'_{\cal S}}), \ell,\ell')$ is an instance of Theorem \ref{bboboiboreach} satisfying  the conditions \eqref{ellell},  \eqref{jbuoouou},  \eqref{GYcond}, and  \eqref{dknjkbduzvreach} for $\hat S.$ Thus, by  Theorem \ref{bboboiboreach}, there exists an ${\sf M}'_{\cal S}$-reachability-based $(\ell,\ell')$-limited packing of  arborescences in $D$. By the construction of ${\sf M}'_{\cal S}$, this provides an ${\sf M}$-reachability-based $(\ell,\ell')$-limited packing of  $S'$-branchings $(S'\subseteq S\in {\cal S})$ in $D$.
\end{proof}

We conclude this section by mentioning an NP-complete result on packing branchings.

\begin{thm} \label{femlbfmejbfkje}
Let $D=(V,A)$ be a digraph, ${\cal S}$ a family of subsets of $V$,  ${\sf M}$ a matroid on ${\cal S}.$ 
 It is NP-complete to decide whether there exists an {\sf M}-based packing of  $S$-branchings $(S\in {\cal S}')$ in $D$ with ${\cal S}'\subseteq {\cal S}$ and $|{\cal S}'|=k,$ even for the uniform matroid of rank $k.$
\end{thm}

\begin{proof}
In the special case when {\sf M} is the uniform matroid of rank $k,$ the problem is equivalent to whether there exists a packing of spanning $S$-branching $(S\in {\cal S}')$ in $D$ with ${\cal S}'\subseteq {\cal S}$ and $|{\cal S}'|=k,$ which is known to be NP-complete, see Theorem 3.6 in B\'erczi and Frank \cite{BF}.
\end{proof}


\section{Packings in directed hypergraphs}

 Let {\boldmath$\mathcal{D}$} $=(V,\mathcal{A})$ be a directed hypergraph, shortly {\it dypergraph}, where {\boldmath$\mathcal{A}$} is the set of dyperedges of $\mathcal{D}.$ A  {\it dyperedge} $e$ is an ordered pair $(Z,z)$, where $z\in V$ is the  {\it head} of $e$ and $\emptyset\neq Z\subseteq V-z$ is the set of {\it tails}  of $e.$  All of our previous results can be extended to directed hypergraphs  from  the corresponding graphic versions, by applying the  gadget of \cite{FKLSzT}.

\section{Proofs}

In this section we provide the main proofs of the paper.

%

\subsection{Proof of Theorem \ref{bboboiboreach1}}\label{pr11}

\begin{proof}
To prove the {\bf necessity}, let {\boldmath${B}_1,\dots,{B}_k$} be an ${\sf M}$-reachability-based packing of  arborescences in $D$ using a subset $A''$  of $A'$ of size at least $\ell$ at most $\ell'$. Then \eqref{ellell} holds. By the monotonicity of $r_{\sf M}$ and  since $N_{A''}^-(v)$ is independent in {\sf M}, we have $$\sum_{v\in V}r_{\sf M}(N_{A'}^-(v))\ge\sum_{v\in V}r_{\sf M}(N_{A''}^-(v))=\sum_{v\in V}d_{A''}^-(v)=|A''|\ge\ell,$$ so \eqref{jzbdjzbzb} also holds.
Let $\mathcal{Z}$ be an $A'$-disjoint subset of $\hat{\mathcal{Z}}.$ By the  necessity of Theorem \ref{thmCsaba} for $(V,A\cup A'')$ and $A''\subseteq A'$, we have  for every $Z\in\mathcal{Z},$ 
	\begin{eqnarray} \label{jbfejvij}
		r_{\sf M}(S\cap P_Z)-r_{\sf M}(S\cap Z)\le d^-_{A\cup A''}(Z)\le d^-_{A\cup A'}(Z),
	\end{eqnarray}
 so \eqref{1dknjkbduzvreachex} holds.
By  \eqref{jbfejvij} and since every arc of $A''\subseteq A'$ enters at most one set in $\mathcal{Z}$, we get that 
$$\sum_{Z\in\mathcal{Z}}(r_{\sf M}(S\cap P_Z)-r_{\sf M}(S\cap Z)-d^-_A(Z))\le\sum_{Z\in\mathcal{Z}}d^-_{A''}(Z)\le |A''|\le\ell',$$ so \eqref{1dknjkbduzvreach} holds.
\medskip

To prove the {\bf sufficiency}, let us suppose that \eqref{ellell}, \eqref{jzbdjzbzb}, \eqref{1dknjkbduzvreachex} and \eqref{1dknjkbduzvreach} hold. The proof relies on the polyhedral description,  obtained from   Theorem \ref{thmGY}(a), of the subgraphs that admit an ${\sf M}$-reachability-based packing of  arborescences in $D$.
First we  focus on the upper bound $\ell'.$ To this purpose we find a subgraph  admitting an ${\sf M}$-reachability-based packing of  arborescences in $D$ that contains the minimum number of arcs in $A'$. 
The lower bound $\ell$ is achieved through Theorem \ref{thmGY}(b). Let us hence consider the following dual linear programs where   {\boldmath$c$}$(a)=1$ if $a\in A'$  and $0$ if $a\in A$:
 
\begin{minipage}{8truecm}
	\begin{eqnarray*}
			x(\delta^-_{A^*}(Z))	&	\ge	&	\hat p(Z) \hskip .4truecmZ\in\hat{\mathcal{Z}}	\\
						     -x(a) 	&	\ge 	&	-1	\hskip .8truecm a\in A^*			\\
		(P)	\hskip 2truecm		  x	&	\ge 	&	 \mathbb{0}						\\
						     c^Tx	&	=	&	w(\min)
	\end{eqnarray*}
\end{minipage}
\begin{minipage}{8truecm}
	\begin{eqnarray*}
			-q(a)+\sum_{a\in\delta^-(Z)}y(Z)			&	\le	&	c(a)   \hskip .5truecm a\in A^*	\\
		(D)		\hskip 3truecm		 y, q 		&	\ge 	&	 \mathbb{0} 				\\
					\hat p^Ty-\mathbb{1}^Tq		&	=	&	z(\max)
	\end{eqnarray*}
\end{minipage}
\medskip

Note that since $S$ is a set of vertices, that is every element of $S$ has multiplicity one, $S\cap {P_Z}=S_{P_Z}$ and $S\cap {Z}=S_{Z}$ for every $Z\subseteq V\cup S.$
Thus, by \eqref{1dknjkbduzvreachex}, the vector $\mathbb{1}$ is a feasible solution of $(P)$ and, by $c\ge 0$, the vector $\mathbb{0}$ is a feasible solution of $(D).$
The complementary slackness theorem says that feasible solutions $\overline x$ and $\overline y\choose\overline q$ of $(P)$ and $(D)$ are optimal if and only if
\begin{eqnarray}
\overline x(a)>0 &&	\Longrightarrow \hskip .55truecm -\overline q(a)+\sum_{a\in\delta^-(Z)}\overline y(Z)	=c(a),\label{cscx}\\
\overline y(Z)>0 &&	\Longrightarrow \hskip 2.45truecm \overline x(\delta^-_{A^*}(Z))=\hat p(Z),\label{cscy}	\\
\overline q(a)>0 &&	\Longrightarrow \hskip 3.45truecm \overline x(a)=1.\label{cscq}
\end{eqnarray}

By  Theorems \ref{mainTDI}  and \ref{EG}, there exist integral optimal solutions {\boldmath$\overline x$} and {\boldmath$\overline y\choose\overline q$} of $(P)$ and $(D)$ that minimizes $\mathbb{1}^T{\overline y\choose\overline q}$ and then that maximizes $\sum\limits_{Z\in\hat{\mathcal{Z}}}|Z|^2\overline y(Z)$. Let {\boldmath$\mathcal{Z}$} $=\{Z\in\hat{\mathcal{Z}}:\overline{y}(Z)>0\}$.  

We want to prove that $c^T\overline x\le \ell'.$ To obtain it we need some properties of $\mathcal{Z}$.
\begin{Lemma}\label{intoptsolprop}
The following hold:
	\begin{enumerate}[(a)]
 \setlength\itemsep{-0cm}
		\item\label{a} if $Z\in\mathcal{Z},$ then there exists an arc $a$  entering $Z$ with $\overline q(a)=0,$
		\item\label{b} if $a\in A^*$ enters $Z\in\mathcal{Z}$ and $\overline{q}(a)=0,$ 
			             then $a\in A'$, $a$ enters no other member of $\mathcal{Z},$ and $\overline{y}(Z)=1,$
		\item\label{c} if $Z\in\mathcal{Z},$ then  $\overline{y}(Z)=1,$ 
		\item\label{d} if $a\in A^*$  enters $Z\in\mathcal{Z}$ and $\overline{q}(a)\neq 0,$ 
			      	     then $a\in A$ and $\overline{q}(a)=\sum\limits_{a\in\delta^-(Z')}\overline y(Z'),$
		\item\label{e} $\mathcal{Z}$ is an $A'$-disjoint core-laminar subset of 
				     $\hat{\mathcal{Z}}$,
		\item\label{f}  if $a\in A^*$  enters no $Z\in\mathcal{Z}$, then  $\overline{q}(a)=0,$
		\item\label{g} $\sum\limits_{a\in A^*}\overline q(a)=\sum\limits_{Z\in\mathcal{Z}}d_A^-(Z),$
		\item\label{h} $c^T\overline x\le \ell'.$
	\end{enumerate}
\end{Lemma}

\begin{proof}
\eqref{a}  Suppose that there exists {\boldmath$Z$} $\in\mathcal{Z}$ such that  $\overline q(a)\neq 0$ for all $a\in\delta_{A^*}^-(Z).$ Let 
\medskip

\centerline{{\boldmath$\overline y'$} $=\overline y-\chi^{\hat{\mathcal{Z}}}_{\{Z\}}$ and {\boldmath$\overline q'$} $=\overline q-\chi^{A^*}_{\delta_{A^*}^-(Z)}.$} 
\medskip

Since $\overline y\in\mathbb{Z}_+^{\hat{\mathcal{Z}}}$ and $Z\in\mathcal{Z}$ (and hence $\overline y(Z)\ge 1$), we have $\overline y'\in\mathbb{Z}_+^{\hat{\mathcal{Z}}}$.  Since $\overline q\in\mathbb{Z}_+^{A^*}$  and $\overline q(a)\neq 0$ (and hence $\overline q(a)\ge 1$) for all $a\in\delta_{A^*}^-(Z)$, we have  $\overline q'\in\mathbb{Z}_+^{A^*}$.  Further, for all $a\in\delta_{A^*}^-(Z)$, we have
$$-\overline q'(a)+\sum_{a\in\delta^-(Z'')}\overline y'(Z'')=-(\overline q(a)-1)+(\sum_{a\in\delta^-(Z'')}\overline y(Z''))-1=-\overline q(a)+\sum_{a\in\delta^-(Z'')}\overline y(Z'').$$ 
Since $\overline x$ and $\overline y\choose\overline q$ are optimal solutions of $(P)$ and $(D),$ it follows that  $\overline y'\choose \overline q'$ is a feasible solution of $(D)$ and that $\overline x$ and ${\overline y'\choose\overline q'}$ satisfy the complementary slackness conditions, so $\overline y'\choose\overline q'$  is an optimal solution of $(D)$. However, $\mathbb{1}^T{\overline y'\choose\overline q'}<\mathbb{1}^T{\overline y\choose\overline q},$ that contradicts the choice of $\overline y\choose\overline q$.
\medskip

\eqref{b}  Let {\boldmath$a$} be an arc with $\overline{q}(a)=0$ that enters {\boldmath$Z$} $\in\mathcal{Z}$. Then, since $\overline y\choose \overline q$ is an integral feasible solution of $(D)$ and $c(a)\le 1,$ we have $$0=\overline{q}(a)\ge-c(a)+\sum_{a\in\delta^-(Z')}\overline y(Z')\ge -c(a)+\overline y(Z)\ge -1+1=0,$$ hence equality holds everywhere, thus $\overline y(Z)=1, c(a)=1$ and $\overline y(Z')=0$ for all $Z'\neq Z$ entered by $a$, so \eqref{b}  holds.
\medskip

\eqref{c}  It immediately follows from \eqref{a}  and \eqref{b}. 
\medskip

\eqref{d}  Let {\boldmath$a$} be an arc  with $\overline{q}(a)\neq 0$ that enters {\boldmath$Z$} $\in\mathcal{Z}.$ Then, since $\overline x$ and $\overline y\choose\overline q$ are optimal solutions of $(P)$ and $(D)$, we get,  by \eqref{cscq} and \eqref{cscx}, that $-\overline q(a)+\sum\limits_{a\in\delta^-(Z')}\overline y(Z')=c(a).$ To finish the proof we have to show that $a\in A$ (and hence $c(a)=0$). Suppose  for a contradiction that $a=$ {\boldmath$sv$} $\in A',$ with $s\in S$ and $v\in V$. Let {\boldmath$Z'$} $=Z\cup\{s\}.$ Since $Z\in\hat{\mathcal{Z}}$, we have $Z'\in\hat{\mathcal{Z}}$. Let
\medskip

\centerline{{\boldmath$\overline y'$} $=\overline y-\chi^{\hat{\mathcal{Z}}}_{\{Z\}}+\chi^{\hat{\mathcal{Z}}}_{\{Z'\}}$ and {\boldmath$\overline q'$} $=\overline q-\chi^{A^*}_{a}.$}
\medskip

Since $\overline y\in\mathbb{Z}_+^{\hat{\mathcal{Z}}}$ and $Z\in\mathcal{Z}$, we have $\overline y'\in\mathbb{Z}_+^{\hat{\mathcal{Z}}}$.  As $\overline q\in\mathbb{Z}_+^{A^*}$  and $\overline q(a)\neq 0$, we have  $\overline q'\in\mathbb{Z}_+^{A^*}$. Since no arc enters $s$, only $a$ leaves $s$, $\overline{q}(a)>0$ and \eqref{cscq}, we get 
\begin{eqnarray}\label{fefdvdvfbg}
 \overline x(\delta_{A^*}^-(Z'))=\overline x(\delta_{A^*}^-(Z))-\overline x(a)= \overline x(\delta_{A^*}^-(Z))-1.
 \end{eqnarray}
Since $Z'\in\hat{\mathcal{Z}}$, $\overline x$ is a feasible solution of $(P)$, by \eqref{fefdvdvfbg},  $Z\in\mathcal{Z}$, \eqref{cscy}, and $r_{\sf M}(S\cap Z')\le r_{\sf M}(S\cap Z)+1,$ we have 
$$\hat p(Z')\le \overline x(\delta_{A^*}^-(Z'))= \overline x(\delta_{A^*}^-(Z))-1=\hat p(Z)-1\le \hat p(Z').$$
It follows that equality holds everywhere, that is $\hat p(Z')=\overline x(\delta_{A^*}^-(Z')).$ 

For all  $b\in \delta_{A^*}^-(Z'),$ 
$$-\overline q'(b)+\sum_{b\in\delta^-(Z'')}\overline y'(Z'')=-\overline q(b)+(\sum_{b\in\delta^-(Z'')}\overline y(Z''))-1+1=-\overline q(b)+\sum_{b\in\delta^-(Z'')}\overline y(Z'').$$

For $a,$ 
$$-\overline q'(a)+\sum_{a\in\delta^-(Z'')}\overline y'(Z'')=-(\overline q(a)-1)+(\sum_{a\in\delta^-(Z'')}\overline y(Z''))-1=-\overline q(a)+\sum_{a\in\delta^-(Z'')}\overline y(Z'').$$ 

Since  $\overline x$ and $\overline y\choose\overline q$ are optimal solutions of $(P)$ and $(D)$, the above arguments show that $\overline y'\choose\overline q'$ is a feasible solution of $(D)$ and $\overline x$ and $\overline y'\choose \overline q'$ satisfy the complementary slackness conditions, so  $\overline y'\choose\overline q'$ is an optimal solution of $(D)$. However, $\mathbb{1}^T{\overline y'\choose\overline q'}<\mathbb{1}^T{\overline y\choose\overline q}$  contradicts the choice of $\overline y\choose\overline q$.
\medskip

\eqref{e}  It immediately follows from \eqref{b} and \eqref{d}  that $\mathcal{Z}$ is an $A'$-disjoint subset of $\hat{\mathcal{Z}}$. We prove by the usual uncrossing technique that $\mathcal{Z}$ is core-laminar. Suppose for a contradiction that there exist core-intersecting $Z_1$ and $Z_2$ in $\mathcal{Z}$ such that  $Z_1-Z_2\neq\emptyset$ and $Z_2-Z_1\neq\emptyset.$ Let {\boldmath$\overline y'$} $=\overline y-\chi^{\hat{\mathcal{Z}}}_{\{Z_1\}}-\chi^{\hat{\mathcal{Z}}}_{\{Z_2\}}+\chi^{\hat{\mathcal{Z}}}_{\{Z_1\cap Z_2\}}+\chi^{\hat{\mathcal{Z}}}_{\{Z_1\cup Z_2\}}.$ Since $\overline y\in\mathbb{Z}_+^{\hat{\mathcal{Z}}}$ and $Z_1,Z_2\in\mathcal{Z}$, we have $\overline y'\in\mathbb{Z}_+^{\hat{\mathcal{Z}}}$. Since $\overline y\choose\overline q$ is a feasible solution of  $(D)$ and $\sum\limits_{a\in\delta^-(Z)}\overline y'(Z)\le \sum\limits_{a\in\delta^-(Z)}\overline y(Z),$ $\overline y'\choose\overline q$ is also a feasible solution of  $(D).$ Then,  since $\overline y\choose\overline q$ is an optimal solution of  $(D),$ $Z_1$ and $Z_2$ are core-intersecting, and by Claim \ref{psupmod}(a), we have  $$0\le (\hat p^T\overline y-\mathbb{1}^T\overline q)-(\hat p^T\overline y'-\mathbb{1}^T\overline q)=\hat p(Z_1)+\hat p(Z_2)-\hat p(Z_1\cap Z_2)-\hat p(Z_1\cup Z_2)\le 0,$$ so $\overline y'\choose\overline q$ is an optimal solution of  $(D).$ Note that $\mathbb{1}^T{\overline y'\choose\overline q}=\mathbb{1}^T{\overline y\choose\overline q}.$ However, by $Z_1-Z_2\neq\emptyset\neq Z_2-Z_1$, we have  $$\sum_{Z\in\hat{\mathcal{Z}}}|Z|^2\overline y(Z)-\sum_{Z\in\hat{\mathcal{Z}}}|Z|^2\overline y'(Z)=|Z_1|^2+|Z_2|^2-|Z_1\cap Z_2|^2-|Z_1\cup Z_2|^2<0$$ that contradicts the choice of ${\overline y\choose\overline q}$. 
\medskip

\eqref{f}  Suppose that there exists an arc {\boldmath$a$} with $\overline{q}(a)\neq 0$ that enters no $Z\in\mathcal{Z}.$ Since $\overline q(a)\ge 0,$ we have $\overline q(a)>0.$ Then, by \eqref{cscq},  \eqref{cscx}, and  $c(a)\ge 0,$ we have  a contradiction: 
$$0=\sum_{a\in\delta^-(Z)}\overline y(Z)=c(a)+\overline q(a)>0+0.$$

\eqref{g}  By \eqref{f}, \eqref{d}, and \eqref{c},  we have $$\sum_{a\in A^*}\overline q(a)=\sum\limits_{\substack{a\in A^*\\\overline q(a)>0}}\overline q(a)=\sum\limits_{\substack{a\in A^*\\\overline q(a)>0}} \sum_{a\in\delta_A^-(Z)}\overline y(Z)=\sum\limits_{\substack{a\in A^*\\\overline q(a)>0}}\sum\limits_{\substack{a\in\delta^-_A(Z)\\Z\in\mathcal{Z}}}1=\sum_{Z\in\mathcal{Z}}\sum_{a\in\delta^-_A(Z)}1=\sum_{Z\in\mathcal{Z}}d_A^-(Z).$$

 \eqref{h}  Since $\overline x$ and $\overline y\choose\overline q$ are optimal solutions  of $(P)$ and $(D),$ we have, by strong duality, that $c^T\overline x=\hat p^T\overline y-\mathbb{1}^T\overline q$, which, by  \eqref{c} and \eqref{g}, is equal to $\sum_{Z\in\mathcal{Z}}(r_{\sf M}(S\cap P_Z)-r_{\sf M}(S\cap Z)-d^-_A(Z))$. By \eqref{e} and \eqref{1dknjkbduzvreach}, this sum is at most $\ell'.$
\end{proof}

\noindent Let {\boldmath$A_1$} $=\{a\in A^*: \overline x(a)=1\}.$ 
Note that, by Lemma \ref{intoptsolprop}\eqref{h}, we have
	\begin{eqnarray}\label{A1ell}
		|A'\cap A_1|=c^T\overline x\le \ell'.
	\end{eqnarray}
Note that $N_{A_1\cap A'}^-(v)$ is independent in ${\sf M}$ for all $v\in V.$ Indeed, let $D(A_1)=(V\cup S,A_1).$ Since $\overline x$ is a feasible solution of $(P)$, \eqref{GYcond2} holds for $(D(A_1),S,{\sf M}).$ Then, by Theorem \ref{thmGY}(b),  there exists an ${\sf M}$-reachability-based packing $\mathcal{B}_1$ of arborescences in $D(A_1)$. Hence $\chi^{A^*}_{A(\mathcal{B}_1)}$ is a feasible solution of $(P).$ Since $\overline x$ is an optimal solution of $(P)$, $|A_1\cap A'|=c^T\overline x\le c^T\chi^{A^*}_{A(\mathcal{B}_1)}=|A(\mathcal{B}_1)\cap A_1\cap A'|$ so $A_1\cap A'\subseteq A(\mathcal{B}_1).$ Then, for each  $v\in V$, the set $R_v$ of roots of the arborescences in $\mathcal{B}_1$ containing $v$ contains $N^-_{A_1\cap A'}(v)$. Since $\mathcal{B}_1$ is  an ${\sf M}$-reachability-based packing $\mathcal{B}_1$ of arborescences, $R_v$ is independent in ${\sf M}$ and thus   $N^-_{A_1\cap A'}(v)$  is independent in ${\sf M}$.

Let {\boldmath$A_2$} be obtained  by adding to  $A_1$ a smallest arc set in $A'$ such that $N_{A_2}^-(v)$ is independent in {\sf M} and $\sum_{v\in V}|N_{A_2}^-(v)|\ge\ell.$ By \eqref{jzbdjzbzb}, this arc set exists. Let   {\boldmath$D'$}  $=(V,A_1\cap A),$  {\boldmath$S'$} the multiset of $V$ such that $S'_v=N_{A_2}^-(v)$ for every $v\in V$ and {\boldmath${\sf M}'$} the restriction of {\sf M} on $S'.$ Let us check that the conditions of Theorem \ref{thmGY}(b) are satisfied for $(D',S',{\sf M}')$. First observe that
	\begin{eqnarray}
		S'_v	\in	 \mathcal{I}_{{\sf M}'} \hskip .4truecm\text{ for every } v\in V.\label{gzddz}
	\end{eqnarray}
For $Z'\in\hat{\mathcal{Z}}_{D'},$ let $Z=Z'\cup (S\cap N_{A_2}^-(Z')).$ 
By the definition of $Z$ and the assumption on $S,$ we have $d_{A_1\cap A}^-(Z')=d^-_{A_1}(Z)	=\overline x(\delta^-_{A^*}(Z)).$
Note that $Z'\in\hat{\mathcal{Z}}_{D'}$ implies that $Z\in\hat{\mathcal{Z}}_{D}.$
Since $\overline x$ is a feasible solution of $(P),$ $\overline x(\delta^-_{A^*}(Z))	\ge	r_{\sf M}(S\cap P^D_Z)-r_{\sf M}(S\cap Z).$ By the constructions of $D'$ and $Z$, we have $r_{\sf M}(S\cap P^D_Z)\ge r_{{\sf M}'}(S'_{P^{D'}_{Z'}})$ and $r_{\sf M}(S\cap Z)	=r_{{\sf M}'}(S'_{Z'}).$ Hence we have 
\begin{eqnarray}
	d_{A_1\cap A}^-(Z')	=\overline x(\delta^-_{A^*}(Z))	\ge	r_{\sf M}(S\cap P^D_Z)-r_{\sf M}(S\cap Z)	\ge	r_{{\sf M}'}(S'_{P^{D'}_{Z'}})-r_{{\sf M}'}(S'_{Z'})  \hskip .4truecm \forall Z'\in\hat{\mathcal{Z}}_{D'}.\label{iuvyttr}
\end{eqnarray}
By \eqref{gzddz} and \eqref{iuvyttr}, the conditions of Theorem \ref{thmGY}(b) are satisfied for $(D',S',{\sf M}')$,  and hence there exists a complete ${\sf M}'$-reachability-based packing $\mathcal{B}'$ of arborescences in $D'$. 
By adding to each $s'$-arborescence in $\mathcal{B}'$  the arc from the corresponding vertex of $S$ to $s'$ and adding to the packing each other vertex $s$ of $S$ with $r_{\sf M}(s)=1$ as an arborescence, we obtain a packing $\mathcal{B}$  of arborescences in $D$. 
Since $\mathcal{B}'$ is an ${\sf M}'$-reachability-based packing of arborescences in $D'$, 
\begin{eqnarray}\label{kjgjhchgdgf}
|R^{\mathcal{B}'}_v|=r_{{\sf M}'}(S'_{P_v^{D'}}) \text{ for every } v\in V.
\end{eqnarray}
Since $\mathcal{B}'$ is complete,  $\mathcal{B}$  uses all the arcs in $A_2\cap A'.$ Further, $|A_2\cap A'|$  is  equal to either $\ell$ and so, by \eqref{ellell}, is at most $\ell'$ or  $|A_1\cap A'|$ which is then  at least $\ell$ and, by \eqref{A1ell}, at most $\ell'.$

For every $v\in V,$ let $Z'_v=P_v^{D'}$ and $Z_v=Z'_v\cup (S\cap N_{A_2}^-(Z'_v))$. 
We show that $Z'_v\in\hat{\mathcal{Z}}_{D'}$. Indeed, let $C'_v$ be the strongly-connected component of $D'$ containing $v.$ Then $Z'_v\cap C'_v\neq\emptyset, Z'_v=P^{D'}_{C'}$ and $d^-_{A_1\cap A}(Z'_v-C'_v)=0,$ so 
$Z'_v\in\hat{\mathcal{Z}}_{D'}$. As above, this implies that $Z_v\in\hat{\mathcal{Z}}_{D}.$
Note also that $v\in Z_v$ and $S\cap Z_v=S\cap N_{A_2}^-(Z'_v).$
Then, by \eqref{iuvyttr} applied for $Z'_v$ and the monotonicity of $r_{{\sf M}}$, we have 
\begin{eqnarray*}
0=d_{A_1\cap A}^-(Z'_v)\ge r_{\sf M}(S\cap P^D_{Z_v})-r_{\sf M}(S\cap Z_v)\ge r_{\sf M}(S\cap P^D_{v})-r_{\sf M}(S\cap N_{A_2}^-(P_v^{D'}))\ge 0.
\end{eqnarray*}
Hence equality holds everywhere, so $r_{\sf M}(S\cap P^D_{v})=r_{\sf M}(S\cap N_{A_2}^-(P_v^{D'})).$ Further,  by the construction of $D'$, we have $r_{\sf M}(S\cap N_{A_2}^-(P_v^{D'}))=r_{{\sf M}'}(S'_{P_v^{D'}}).$ Hence, by the construction of $\mathcal{B}$ and \eqref{kjgjhchgdgf}, we get that $|R^{\mathcal{B}}_v|=|R^{\mathcal{B}'}_v|=r_{{\sf M}'}(S'_{P_v^{D'}})=r_{\sf M}(S\cap P^D_{v})$ for every $v\in V.$ Further, $|R^{\mathcal{B}}_s|=r_{\sf M}(s)=r_{\sf M}(S\cap P^D_{s})$ for every $s\in S,$ so the packing $\mathcal{B}$  is ${\sf M}$-reachability-based.  
\end{proof}

\subsection{Proof of Theorem \ref{bboboiboreach}}\label{pr12}

\begin{proof} 
To prove the {\bf necessity}, take an ${\sf M}$-reachability-based packing $\mathcal{B}$ of $s$-arborescences  $(s\in S^*)$  in $D$ for some {\boldmath$S^*$} $\subseteq S$ with $\ell\le |S^*|\le\ell'$. Then, clearly, \eqref{ellell} holds.
Also holds \eqref{jbuoouou}  because, since $r_{\sf M}$ is non-decreasing and $S^*_v$ is independent in ${\sf M},$ we have  
 \begin{eqnarray*}
\sum_{v\in V}r_{\sf M}(S_v)\ge \sum_{v\in V}r_{{\sf M}}(S^*_v)=\sum_{v\in V}|S^*_v|=|S^*|\ge\ell.
 \end{eqnarray*}

By the necessity of Theorem \ref{thmGY}, \eqref{GYcond} also holds.
Let  {\boldmath${\sf X}$} $\in {\mathcal X}$ and $v\in X_I.$ Let $R^v=R^{\mathcal{B}}_v.$
\begin{eqnarray}
	|R^v|	=|R^v_{X_I}|+|R^v_{X_W}|+|R^v_{\overline {X_O}}|.	\label{eryyjukure}
 \end{eqnarray}
Since the packing is ${\sf M}$-reachability-based, $R^v$ is a base of $S_{P_v}.$ Since $v\in X_I\subseteq C_{\sf X}$ and $C_{\sf X}$ is strongly-connected, we have $P_v\subseteq P_{X_I}\subseteq P_{C_{\sf X}}\subseteq P_v$. Then 
\begin{eqnarray}
	|R^v|	=r_{\sf M}(S_{P_{v}})=r_{\sf M}(S_{P_{X_I}}).	\label{nkmnjbhvre}
 \end{eqnarray}
Since $R^v_{X_W}\subseteq R^v$, $R^v$ is independent in {\sf M} and $r_{\sf M}$ is monotone, we have 
\begin{eqnarray}
	|R^v_{X_W}|=r_{\sf M}(R^v_{X_W})\le r_{\sf M}(S_{X_W}).	\label{lknvytxtrre}
 \end{eqnarray}
Since the arborescences are arc-disjoint and $v\in X_O$, we have 
\begin{eqnarray}
	|R^v_{\overline {X_O}}|\le d_A^-(X_O).	\label{lknouhvyftxere}
 \end{eqnarray}
It follows from \eqref{eryyjukure}--\eqref{lknouhvyftxere} that 
\begin{eqnarray}
	|S^*_{X_I}|\ge |R^v_{X_I}|\ge r_{\sf M}(S_{P_{X_I}})-r_{\sf M}(S_{X_W})-d_A^-(X_O).	\label{kjbouyxreytre}
 \end{eqnarray}
 Let ${\cal P}$ be an OW laminar biset family of $\mathcal{X}.$ Then,  by $\ell'\ge |S^*|$, since $X_I$'s are disjoint,  and  by \eqref{kjbouyxreytre}, we have  
 \begin{eqnarray*}
\ell'\ge |S^*|\ge\sum_{{\sf X}\in{\cal P}}|S^*_{X_I}|\ge \sum_{{\sf X}\in{\cal P}}(r_{\sf M}(S_{P_{{X}_I}})-r_{\sf M}(S_{X_W})-d_A^-(X_O)),
 \end{eqnarray*}
  so \eqref{dknjkbduzvreach} holds. 
\medskip

To prove the {\bf sufficiency}, suppose that for an instance $(D=(V,A), S, \ell, \ell', {\sf M}=(S,r_{\sf M}))$  of Theorem \ref{bboboiboreach}, the conditions \eqref{ellell}, \eqref{jbuoouou}, \eqref{GYcond} and \eqref{dknjkbduzvreach} hold. 
To be able to apply Theorem \ref{bboboiboreach1} we have to consider the extended version.
Let {\boldmath$D'$} $=(V\cup S',A\cup A')$ hence be obtained from $D$ by adding a new vertex set {\boldmath$S'$} containing one vertex $s'$ for every $s\in S$ and adding a new arc set {\boldmath$A'$} containing one arc $s's$ for every $s\in S$. Let {\boldmath${\sf M}'$} be a copy of {\sf M} on $S'.$ Note that $A'$ is the set of arcs leaving $S'$, no arc enters $s'$ and exactly one arc leaves $s'$ for every vertex $s'$ of $S'.$  Then $(D', \ell, \ell', {\sf M}')$ is an instance of Theorem  \ref{bboboiboreach1}. We now show that all the conditions of Theorem \ref{bboboiboreach1} hold. First, \eqref{ellell} holds by assumption. Note that \eqref{jbuoouou} implies \eqref{jzbdjzbzb} and \eqref{GYcond} implies \eqref{1dknjkbduzvreachex}. Finally, the following lemma  implies that \eqref{1dknjkbduzvreach} also holds.

\begin{Lemma}\label{condholds}
$\sum_{Z\in\mathcal{Z}}(r_{{\sf M}'}(S'\cap P^{D'}_Z)-r_{{\sf M}'}(S'\cap Z)-d_A^-(Z))\le\ell' $  for all $A'$-disjoint core-laminar subset $\mathcal{Z}$ of $\hat{\mathcal{Z}}_{D'}.$
\end{Lemma}

\begin{proof}
 Let  {\boldmath$\mathcal{Z}$} be an $A'$-disjoint core-laminar subset $\mathcal{Z}$ of $\hat{\mathcal{Z}}_{D'}.$ 
 For every $Z^i\in\mathcal{Z}$, we may suppose without loss of generality that $r_{{\sf M}'}(S'\cap P^{D'}_{Z^i})-r_{{\sf M}'}(S'\cap Z^i)-d_A^-(Z^i)\ge 1$ and hence, by \eqref{1dknjkbduzvreachex}, that $d_{A'}^-(Z^i)\ge 1$. Let {\boldmath${\sf X}^i$} be the biset on $V$ with {\boldmath$X^i_O$} $=V\cap Z^i$ and {\boldmath$X^i_I$} $=N^+_{A'}(S'-Z^i)\cap Z^i$ and {\boldmath$\mathcal{P}$} $=\{{\sf X}^i:Z^i\in\mathcal{Z}\}.$ 
 
\begin{Proposition}\label{disjointbiset}
$\mathcal{P}$ is an OW laminar biset family of $\mathcal{X}.$
\end{Proposition}

\begin{proof}
To show that $\mathcal{P}\subseteq {\mathcal{X}}$, let ${\sf X}^i\in\mathcal{P}.$ Since $Z^i\in\mathcal{Z}\subseteq \hat{\mathcal{Z}}_{D'},$ there exists an atom $C$ of $D'$ such that $Z^i\cap C\neq\emptyset, Z^i\subseteq P_C^{D'}$ and $d_{A\cup A'}^-(Z^i-C)=0.$ Then, since $d_{A'}^-(Z^i)\ge 1$, $C$ is an atom of $D,$ $\emptyset\neq X^i_I\subseteq C$, $X^i_W\subseteq P^D_C$ and $d_A^-(X_W-C)=0,$ so we have ${\sf X}^i\in {\mathcal{X}}.$

We now  show that  $\mathcal{P}$ is OW laminar. Suppose there exist two core-intersecting bisets ${\sf X}^i, {\sf X}^j$ in $\mathcal{P}$ such that $X^i_O-X^j_W\neq\emptyset\neq X^j_O-X^i_W.$ Then $Z^i$ and $Z^j$ are $\hat{\mathcal{Z}}_{D'}$-intersecting. Since $\mathcal{Z}$ is $D'$-core-laminar, we get that $Z^i\subseteq Z^j$ or $Z^j\subseteq Z^i,$ say $Z^i\subseteq Z^j$, so $X^i_O\subseteq X^j_O.$ Since $X^i_O-X^j_W\neq\emptyset,$ it follows that there exists a vertex $v\in X^i_O\cap X^j_I.$ Then,  by definition, there exist $s'_{j}\in S'\setminus Z^j$ such that $s'_{j}v\in A'.$ Since $v\in X^i_O\cap X^j_I\subseteq Z^i\subseteq Z^j$ and $s'_{j}\in S'\setminus Z^j\subseteq S'\setminus Z^i$, we get that the arc $s'_{j}v\in A'$ enters $Z^i$ and $Z^j$, that contradicts the fact that $\mathcal{Z}$ is $A'$-disjoint. 
\end{proof}

Note that, by construction of $D',$ we have $S'\cap P^{D'}_{Z^i}=S_{P^D_{X_O}}.$ Since there exists $v\in X_I\subseteq C_{\sf X}$ and $C_{\sf X}$ is strongly-connected, we have $P^D_v\subseteq P^D_{X_I}\subseteq P^D_{X_O}\subseteq P^D_{C_{\sf X}}\subseteq P^D_v$. Then 
\begin{eqnarray} \label{jbfheknfibf}
r_{{\sf M}'}(S'\cap P^{D'}_{Z^i})=r_{{\sf M}}(S_{P^{D}_{X^i_I}}).
\end{eqnarray}
  For every $s\in S_{X^i_W},$ by the definition of $X^i_W,$ we have $s'\in S'\cap Z^i.$ Hence the elements of $S'$ corresponding to $S_{X^i_W}$ are contained in $S'\cap Z^i.$ Then, by the monotonicity of $r_{\sf M},$ we have 
 \begin{eqnarray} \label{lkblljbvuu}
  r_{{\sf M}}(S_{X^i_W})\le r_{{\sf M}'}(S'\cap Z^i).
 \end{eqnarray}
 Since no arc enters $Z^i\cap S'$, we have 
 \begin{eqnarray} \label{nefbefnp}
d_A^-(Z^i)=d_A^-(X^i_O).
\end{eqnarray}
  Thus, by \eqref{jbfheknfibf}--\eqref{nefbefnp}, Proposition \ref{disjointbiset}  and \eqref{dknjkbduzvreach},  we get that $$\sum_{Z^i\in\mathcal{Z}}(r_{{\sf M}'}(S'\cap P^{D'}_{Z^i})-r_{{\sf M}'}(S'\cap Z^i)-d_A^-(Z^i))\le \sum_{{\sf X}^i\in\mathcal{P}}(r_{{\sf M}}(S_{P^{D}_{X^i_I}})-r_{{\sf M}}(S_{X^i_W})-d_A^-(X^i_O))\le\ell',$$ and the proof of the lemma is completed.
\end{proof}

By Theorem \ref{bboboiboreach1},  there exists an ${\sf M}'$-reachability-based packing of  arborescences in $D'$ using at least $\ell$ and  at most $\ell'$ arcs of $A'.$ By deleting the roots $s'$ of the arborescences in the packing, we obtain $s$-arborescences in $D.$ Hence we get an ${\sf M}$-reachability-based packing of at least $\ell$ and at most $\ell'$ arborescences in $D$ that completes the proof of Theorem \ref{bboboiboreach}.
\end{proof}

\section{Algorithmic aspects}

In this section we assume that a matroid is given by an oracle for the rank function and, under this assumption, we point out that all the problems which derive from the problem of Theorem \ref{bboboiboreach1} can be solved in polynomial time. Since all the reductions we presented are done in polynomial time, it is enough to show that the  problem of Theorem \ref{bboboiboreach1} can be solved in polynomial time. The main algorithmic difficulty in the proof of Theorem \ref{bboboiboreach1} is to find an arc set containing a minimum number of arcs leaving $S$, that admits a matroid-reachability-based packing of arborescences.  This can be done by finding the arc set  of a matroid-reachability-based packing of arborescences of minimum weight where the weight is  $1$ for every arc leaving $S$ and $0$ for the other arcs. This latter problem can be solved in polynomial time due to {B\'erczi,  Kir\'aly, Kobayashi \cite{BKK} or Kir\'aly, Szigeti, and Tanigawa \cite{KSzT}. Then, we mention that the last part in the proof of Theorem \ref{bboboiboreach1}, that is finding a complete matroid-reachability-based packing of arborescences, is polynomial (see Kir\'aly \cite{cskir}, H\"orsch and Szigeti \cite{HSz4}) hence we can get the required packing in polynomial time.

\appendix

\section{Detailed results on packings in directed hypergraphs}

The aim of this section is to present the extensions of the previous results to directed hypergraphs. We start with the necessary definitions on directed hypergraphs.
\medskip

 Let {\boldmath$\mathcal{D}$} $=(V,\mathcal{A})$ be a directed hypergraph, shortly {\it dypergraph}, where {\boldmath$\mathcal{A}$} is the set of dyperedges of $\mathcal{D}.$ A  {\it dyperedge} $e$ is an ordered pair $(Z,z)$, where $z\in V$ is the  {\it head} of $e$ and $\emptyset\neq Z\subseteq V-z$ is the set of {\it tails}  of $e.$ For a  subset $X$ of $V,$ a dyperedge $(Z,z)$ {\it enters $X$} if $z\in X$ and $Z-X\neq\emptyset.$ The {\it in-degree} {\boldmath$d^-_\mathcal{A}(X)$} of $X$ is the number  of dyperedges in $\mathcal{A}$ {\it entering $X$}. For a subpartition ${\cal P}$ of $V$, we denote by {\boldmath$e_{\mathcal{A}}({\cal P})$} the set of dyperedges in $\mathcal{A}$ that enters at least one member of ${\cal P}$.
 By {\it trimming} a dyperedge $e=(Z,z)$, we mean the operation that replaces $e$ by an arc $yz$ where $y\in Z.$  A dypergraph $\mathcal{D}$ is called a {\it (spanning) $s$-hyperarborescence} if $\mathcal{D}$ can be trimmed to a (spanning) $s$-arborescence. A dypergraph $\mathcal{D}$ is called a {\it dyperpath from $s$ to $t$} if $\mathcal{D}$ can be trimmed to a path from $s$ to $t.$ 
 A {\it subatom} of $\mathcal{D}$ is a non-empty subset $C$ of vertices such that for every ordered pair $(u,v)\in C\times C,$ there exists a dyperpath from $u$ to $v$ in $\mathcal{D}$. 
An {\it atom} of $\mathcal{D}$ is a maximal subatom of $\mathcal{D}$. 
 For a subset  $X$ of $V,$ we denote by {\boldmath$P^\mathcal{D}_X$}  the set of vertices from which there exists a dyperpath to  at least one vertex of $X$.
 \medskip

 Let $S$ be a multiset of $V$ and ${\sf M}$ a matroid on $S.$ A packing $\mathcal{B}$ of hyperarborescences in $\mathcal{D}$  is called {\sf M}-{\it based} or {\it  matroid-based} if every $s\in S$ is the root of at most one hyperarborescence in the packing and for every vertex $v\in V$, the multiset {\boldmath$R^{\mathcal{B}}_v$} of roots of hyperarborescences in the packing in which $v$ can be reached from the root forms a basis of ${\sf M}.$ A packing $\mathcal{B}$ of hyperarborescences in $\mathcal{D}$  is called {\sf M}-{\it reachability-based} or {\it matroid-reachability-based}  if every $s\in S$ is the root of at most one hyperarborescence in the packing and for every vertex $v\in V$, the multiset $R^{\mathcal{B}}_v$  forms a basis of $S_{P^\mathcal{D}_v}$ in ${\sf M}.$ A packing of hyperarborescences is {\it complete} if every $s\in S$ is the root of exactly one hyperarborescence in the packing.

\subsection{Packing of  hyperarborescences}

Theorem \ref{edmondsarborescencesmulti} was generalized to dypergraphs in \cite{fkiki}.

\begin{thm}[Frank, Kir\'aly, Kir\'aly  \cite{fkiki}]\label{hyperarborescencesmulti} 
Let $\mathcal{D}=(V,\mathcal{A})$ be a dypergraph and $S$ a multiset of vertices in $V.$
There exists a packing of spanning $s$-hyperarborescences  $(s\in S)$ in $\mathcal{D}$ if and only if 
	\begin{eqnarray} \label{fkkcondmulti} 
		|S_{X}|+ d^-_\mathcal{A}(X) &\geq &|S| \hskip .5truecm \text{ for every  $\emptyset\neq X\subseteq V.$}
	\end{eqnarray}
\end{thm}

If $\mathcal{D}$ is a digraph, then Theorem \ref{hyperarborescencesmulti} reduces to Theorem \ref{edmondsarborescencesmulti}.
\medskip

The following common extension of Theorems \ref{reach1}  and \ref{hyperarborescencesmulti} was given in \cite{BF2}.

\begin{thm}[B\'erczi, Frank \cite{BF2}]\label{reach1hyp} 
Let $\mathcal{D}=(V,\mathcal{A})$ be a dypergraph and $S$ a multiset of vertices in $V.$
There exists a packing of reachability $s$-hyperarborescences $(s\in S)$ in $\mathcal{D}$ if and only if 
	\begin{eqnarray} \label{reach1condhyp} 
		|S_{X}|+ d^-_{\mathcal{A}}(X) &\geq& |S_{P^{\mathcal{D}}_X}| \hskip .5truecm \text{ for every  $X\subseteq V.$}
	\end{eqnarray}
\end{thm}

If $\mathcal{D}$ is a digraph, then Theorem \ref{reach1hyp} reduces to Theorem \ref{reach1}.
The same way as Theorem \ref{reach1} implies Theorem \ref{edmondsarborescencesmulti}, Theorem \ref{reach1hyp} implies Theorem \ref{hyperarborescencesmulti}.

\subsection{Matroid-based packing of hyperarborescences}

Theorem  \ref{thmddgnsz} was generalized to dypergraphs in \cite{FKLSzT} as follows. It was obtained from the  graphic version, Theorem  \ref{thmddgnsz},  by a simple gadget.

\begin{thm}[Fortier, Kir\'aly, L\'eonard, Szigeti, Talon \cite{FKLSzT}] \label{kjvkv}
Let $\mathcal{D}=(V,\mathcal{A})$ be a  dypergraph, $S$ a multiset of vertices in $V$, and ${\sf M}=(S,\mathcal{I}_{\sf M})$ a matroid with rank function $r_{\sf M}$. There exists a complete {\sf M}-based packing of hyperarborescences in $\mathcal{D}$  if and only if   \eqref{matcondori1} holds and 
	\begin{eqnarray}
		d^-_{\mathcal{A}}(Z)\geq r_{{\sf M}}(S)-r_{{\sf M}}(S_Z)&& \text{ for every } Z\subseteq V.\label{lfljfulyflu}
	\end{eqnarray}
\end{thm}

If $\mathcal{D}$ is a digraph, then Theorem \ref{kjvkv} reduces to Theorem \ref{thmddgnsz}.
For the free matroid, Theorem \ref{kjvkv} reduces to Theorem \ref{hyperarborescencesmulti}.
\medskip

Theorem \ref{mbpsaori} was extended to dypergraphs in \cite{szigrooted}.

\begin{thm}[Szigeti \cite{szigrooted}] \label{mbpsaorihyp}
Let $\mathcal{D}=(V,\mathcal{A})$ be a  dypergraph, $S$ a multiset of vertices in $V$, and ${\sf M}=(S,r_{\sf M})$ a matroid. There exists an ${\sf M}$-based packing of  spanning hyperarborescences in $\mathcal{D}$  if and only if 
\begin{eqnarray}
r_{{\sf M}}(S_{\cup{\cal P}})+e_\mathcal{A}({\cal P})&\geq &r_{{\sf M}}(S)|\mathcal{P}| \hskip .5truecm \text{ for every subpartition $\mathcal{P}$ of $V$.}\label{mbpsaoricondhyp}
\end{eqnarray}
\end{thm}

If $\mathcal{D}$ is a digraph, then Theorem \ref{mbpsaorihyp} reduces to Theorem \ref{mbpsaori}.
For the free matroid, Theorem \ref{mbpsaorihyp} reduces to Theorem \ref{kjvkv}. 
\medskip

The following generalizations to dypergraphs follow from  the corresponding graphic versions, by applying the  gadget of \cite{FKLSzT}.

\begin{thm} \label{bboboibohyp}
Let $\mathcal{D}=(V,\mathcal{A})$ be a  dypergraph, $S$ a multiset of vertices in $V$, $\ell, \ell'\in\mathbb{Z}_+,$ and ${\sf M}=(S,r_{\sf M})$ a matroid. There exists an ${\sf M}$-based $(\ell,\ell')$-limited packing of hyperarborescences  in $\mathcal{D}$ if and only if \eqref{ellell} and \eqref{jbuoouou} hold and
\begin{eqnarray} 
	r_{\sf M}(S_{X})+d_{\mathcal{A}}^-(X)	&	\ge	&	r_{\sf M}(S)  \hskip .2truecm \text{ for every subatom } X \text{ of } \mathcal{D},\label{kjbvgchyp}\\
	\sum_{{\sf X}\in\mathcal{P}}(r_{\sf M}(S)-r_{\sf M}(S_{X_W})-d_{\mathcal{A}}^-(X_O)) &	\le	&	\ell'	 
	\hskip 1.15truecm  \forall \text{  OW laminar biset family } {\cal P} \text{ of subatoms}.\hskip .8truecm 	\label{dknjkbduzvhyp}
\end{eqnarray}
\end{thm}

If $\mathcal{D}$ is a digraph, then Theorem \ref{bboboibohyp} reduces to Theorem \ref{bboboibo}.
If $\ell=\ell'=|S|,$ then  Theorem \ref{bboboibohyp} reduces to Theorem \ref{kjvkv}.
If $\ell=\ell'=r_{{\sf M}}(S),$ then  Theorem \ref{bboboibohyp} reduces to Theorem \ref{mbpsaorihyp}.

\begin{thm} \label{vouyfyhyp}
Let $\mathcal{D}=(V,\mathcal{A})$ be a  dypergraph, $S$ a multiset of vertices in $V$, and ${\sf M}=(S,r_{\sf M})$ a matroid. There exists a decomposition of $\mathcal{A}$ into an ${\sf M}$-based packing of hyperarborescences  in $\mathcal{D}$ if and only if \eqref{kjbvgchyp} holds and for every OW laminar biset family ${\cal P}$ of subatoms, 
\begin{eqnarray} \label{ljhfygdtuhyp}
	\sum_{{\sf X}\in\mathcal{P}}(r_{\sf M}(S)-r_{\sf M}(S_{X_W})-d_\mathcal{A}^-(X_O)) 	&	\le	&r_{\sf M}(S)|V|-|\mathcal{A}|. 
\end{eqnarray}
\end{thm}

If $\mathcal{D}$ is a digraph, then Theorem \ref{vouyfyhyp} reduces to Theorem \ref{vouyfy}.

\subsection{Matroid-reachability-based packing of hyperarborescences}

Theorem \ref{thmCsaba} was generalized to dypergraphs in \cite{FKLSzT} as follows. 

 \begin{thm}[Fortier, Kir\'aly, L\'eonard, Szigeti, Talon \cite{FKLSzT}] \label{kjvkv2}
Let $\mathcal{D}=(V,\mathcal{A})$ be a  dypergraph, $S$ a multiset of vertices in $V$, and ${\sf M}=(S,\mathcal{I}_{\sf M})$ a matroid with rank function $r_{\sf M}$.There exists a complete {\sf M}-reachability-based packing of hyperarborescences in $\mathcal{D}$  if and only if   \eqref{matcondori1} holds and 
\begin{eqnarray}
	d^-_{\mathcal{A}}(Z)\geq r_{{\sf M}}(S_{P_Z})-r_{{\sf M}}(S_Z)&& \text{ for every } Z\subseteq V.\label{csabicondhyp}
\end{eqnarray}
\end{thm}

If $\mathcal{D}$ is a digraph, then Theorem \ref{kjvkv2} reduces to Theorem \ref{thmCsaba}.
If $r_{\sf M}(S_{P_v})=r_{\sf M}(S)$ for all $v\in V,$ then  Theorem \ref{kjvkv2} reduces to Theorem \ref{kjvkv}.
\medskip

The following generalizations to dypergraphs follow from  the corresponding graphic versions, by applying the  gadget of \cite{FKLSzT}.

\begin{thm} \label{bboboibohyp2}
Let $\mathcal{D}=(V,\mathcal{A})$ be a  dypergraph, $S$ a multiset of vertices in $V$, $\ell, \ell'\in\mathbb{Z}_+$, and ${\sf M}=(S,r_{\sf M})$ a matroid.  There exists an ${\sf M}$-reachability-based $(\ell,\ell')$-limited packing of  hyperarborescences in $\mathcal{D}$ if and only if \eqref{ellell} and \eqref{jbuoouou}   hold and 
\begin{eqnarray} 
		r_{{\sf M}}(S_{P_{X_I}})-r_{{\sf M}}(S_{X_O}) 	&	\le	&	d_\mathcal{A}^-({X_O}) \hskip .44truecm\text{ for every  biset {\sf X} on $V$,}\label{vvcvclj}  \\
	\sum_{{\sf X}\in\mathcal{P}}(r_{\sf M}(S_{P_{X_I}})-r_{\sf M}(S_{X_W})-d_\mathcal{A}^-(X_O))& \le	&\ell'		\hskip .4truecm  \text{ for every OW laminar biset family } {\cal P} \text{ of }  \mathcal{X}.\hskip .8truecm	
\end{eqnarray}
\end{thm}

If $\mathcal{D}$ is a digraph, then Theorem \ref{bboboibohyp2} reduces to Theorem \ref{bboboiboreach}.
If $r_{\sf M}(S_{P_v})=r_{\sf M}(S)$ for all $v\in V,$ then  Theorem \ref{bboboibohyp2} reduces to Theorem \ref{bboboibohyp}.
If $\ell=\ell'=|S|,$ then  Theorem \ref{bboboibohyp2} reduces to Theorem \ref{kjvkv2}.

\begin{thm} \label{vouyfyhyp2}
Let $\mathcal{D}=(V,\mathcal{A})$ be a  dypergraph, $S$ a multiset of vertices in $V$, and ${\sf M}=(S,r_{\sf M})$ a matroid. There exists a decomposition of $\mathcal{A}$ into an ${\sf M}$-reachability-based packing of hyperarborescences  in $\mathcal{D}$ if and only if \eqref{vvcvclj} holds and for every OW laminar biset family ${\cal P}$ of $\mathcal{X},$ 
\begin{eqnarray} \label{ljhfygdtu2hyp}
\sum_{{\sf X}\in\mathcal{P}}(r_{\sf M}(S_{P_{X_I}})-r_{\sf M}(S_{X_W})-d_\mathcal{A}^-(X_O))& \le	&r_{\sf M}(S)|V|-|\mathcal{A}|.	
\end{eqnarray}
\end{thm}

If $\mathcal{D}$ is a digraph, then Theorem \ref{vouyfyhyp2} reduces to Theorem \ref{vouyfy2}.
If $r_{\sf M}(S_{P_v})=r_{\sf M}(S)$ for all $v\in V,$ then  Theorem \ref{vouyfyhyp2} reduces to Theorem \ref{vouyfyhyp}.


\begin{thebibliography}{99}
\bibitem{BF2} {K. B\'erczi, A. Frank,} \textit{Variations for Lov\'asz' submodular ideas,} in Building Bridges, Springer, (2008) 137--164.
\bibitem{BF} {K. B\'erczi, A. Frank,} \textit{Packing arborescences,}  in: S. Iwata (Ed.), RIMS Kokyuroku Bessatsu B23: Combinatorial Optimization and Discrete Algorithms, Lecture Notes, (2010) 1--31.
\bibitem{BKK} {K. B\'erczi, T. Kir\'aly, Y. Kobayashi}, \textit{Covering intersecting bi-set families under matroid constraints}, SIAM J. Discret. Math.  30(3) (2016) 1758--1774.
\bibitem{Egy} {J. Edmonds}, \textit{Edge-disjoint branchings}, in {Combinatorial Algorithms}, B. Rustin ed., Academic Press, New York, (1973)  91--96.
\bibitem{DdGNSz}  {O. Durand de Gevigney, V. H. Nguyen, Z. Szigeti}, \textit{Matroid-Based Packing of Arborescences}, SIAM J. Discret. Math. 27(1) (2013)  567--574.
\bibitem {TDI} J. Edmonds, R. Giles, \textit{A min-max relation for submodular functions on graphs}, Annals of Discrete Math., 1 (1977) 185--204.
\bibitem{FKLSzT}  {Q. Fortier, Cs. Kir\'aly, M. L\'eonard, Z. Szigeti, A. Talon}, \textit{Old and new results on packing arborescences}, Discret. Appl. Math. 242 (2018) 26--33. 
\bibitem {FRdtda} A. Frank, \textit{On disjoint trees and arborescences}, in: Algebraic Methods in Graph Theory, Colloquia Mathematica Soc. J. Bolyai, 25 (1978) 159--169.
\bibitem {FRCB} A. Frank, \textit{Covering branchings}, Acta Sci. Math. (Szeged) 41 (1--2) (1979) 77--81.
\bibitem {frtdithm} A. Frank, \textit{Rooted $k$-connections in digraphs}, Discrete Appl. Math. 157 (2009) 1242--1254.
\bibitem{book} {A. Frank}, Connections in Combinatorial Optimization, Oxford University Press, 2011.
\bibitem {fkiki} A. Frank, T. Kir\'aly,  Z. Kir\'aly, \textit{On the orientation of graphs and hypergraphs,} Discret. Appl. Math. 131(2) (2003) 385--400.
\bibitem {fkk} A. Frank, T. Kir\'aly, M. Kriesell, \textit{On decomposing a hypergraph into $k$ connected sub-hypergraphs,} Discret. Appl. Math. 131 (2) (2003) 373--383.
\bibitem {GY} H. Gao, D. Yang, \textit{Packing of maximal independent mixed arborescences}, Discrete Appl. Math. 289 (2021) 313--319.
\bibitem{HSz4} {F. H\"orsch, Z. Szigeti,} \textit{Reachability in arborescence packings}, Discret. Appl. Math. 320 (2022) 170--183.
\bibitem{japan}  {N. Kamiyama, N. Katoh, A. Takizawa}, \textit{Arc-disjoint in-trees in directed graphs}, Comb. 29 (2009) 197--214.
\bibitem{KT} { N. Katoh, S. Tanigawa,} \textit{Rooted-tree decomposition with matroid constrains and the infinitesimal rigidity of frameworks with boundaries}, SIAM J. Discret. Math. 27(1), (2013) 155--185. 
\bibitem {cskir} {Cs. Kir\'aly}, \textit{On maximal independent arborescence packing}, SIAM J. Discret. Math. 30(4) (2016) 2107--2114.
\bibitem{KSzT}{Cs. Kir\'aly, Z. Szigeti, S. Tanigawa,} \textit{Packing of arborescences with matroid constraints via matroid intersection}, {Math. Program.} {181} (2020) 85--117.
\bibitem{NW} C. St. J. A. Nash-Williams, \textit{Edge-disjoints spanning trees of finite graphs}, Journal of the London Mathematical Society, 36 (1961) 445--450.
\bibitem{NW2}  C. St. J. A. Nash-Williams, \textit{Decomposition of finite graphs into forests}, J. London Math. Soc., 39 (1964) 12.
\bibitem {lexbook} A. Schrijver, Theory of linear and integer programming, Wiley (1998)
\bibitem {pmh} Z. Szigeti, \textit{Packing mixed hyperarborescences},  Discrete Optimization 50 (2023) 100811.
\bibitem{szigrooted}{Z. Szigeti,} \textit{Matroid-rooted packing of arborescences}, submitted
\bibitem{tay} T. Tay, \textit{Rigidity of multi-graphs I: Linking rigid bodies in n-space}, Journal of Combinatorial Theory. Series B, 36(1) (1984) 95--112.
\bibitem{Tu} W.T. Tutte, \textit{On the problem of decomposing a graph into $n$ connected factors}, Journal of the London Mathematical Society, 36 (1961) 221--230.
\end{thebibliography}
\end{document}